\documentclass[11pt]{amsart}

\usepackage{amsthm} 
\usepackage{amssymb}
\usepackage{graphicx}									 
\usepackage{stmaryrd}                  
\usepackage{footmisc}                  
\usepackage{paralist}
\usepackage{wrapfig}
\usepackage{bbm}
\usepackage[T2A,T1]{fontenc}
\usepackage[utf8]{inputenc}
\usepackage[ngerman,russian,british]{babel}
\usepackage[arrow, matrix, curve]{xy}
\usepackage{color}
\usepackage{pdfcomment}
\usepackage{enumitem}

\newtheorem{thm}{Theorem}[section]

\newtheorem{lem}[thm]{Lemma}

\newtheorem{cor}[thm]{Corollary}
\newtheorem*{corB}{Corollary B}
\newtheorem*{corC}{Corollary C}
\newtheorem{prp}[thm]{Proposition}

\newtheorem*{thmA}{Theorem A}

\theoremstyle{definition}
\newtheorem{dfn}[thm]{Definition}

\theoremstyle{remark}

\newtheorem{rem}[thm]{Remark}

\newcommand{\Co}{\mathbb{C}}
\newcommand{\R}{\mathbb{R}}

\newcommand{\Z}{\mathbb{Z}}

\newcommand{\Orb}{\mathcal{O}}

\newcommand{\alt}{\mathfrak{A}}
\newcommand{\sym}{\mathfrak{S}}
\newcommand{\dih}{\mathfrak{D}}
\newcommand{\cyc}{\mathfrak{C}}

\newcommand{\SOr}{\mathrm{SO}}
\newcommand{\SU}{\mathrm{SU}}

\newcommand{\SL}{\mathrm{SL}}

\newcommand{\Or}{\mathrm{O}}

\newcommand{\To}{\rightarrow}

\newcommand{\MTo}{\mapsto}
\newcommand{\LMTo}{\longmapsto}

\newcommand{\codim}{\text{codim }}
\newcommand{\ico}{\mathbf{I}}
\newcommand{\mini}{\mathrm{min}}
\newcommand{\maxi}{\mathrm{max}}
\newcommand{\rota}{\mathrm{rot}}

\newcommand{\Gp}{G^{\scriptscriptstyle+}}

\newcommand{\subgr}{<}

\title[When is the underlying space of an orbifold a manifold?]{When is the underlying space of an orbifold a manifold?}
\author{Christian Lange}
\thanks{The results of this paper appear in the author's thesis \cite{Lange_thesis}. The author was partially supported by a `Kurzzeitstipendium für Doktoranden' by the German Academic Exchange Service (DAAD) and by the DFG funded project SFB/TRR 191.}
\address{Christian Lange, Mathematisches Institut der Universit\"at zu K\"oln, Weyertal 86-90, 50931 K\"oln, Germany}
\email{clange@math.uni-koeln.de}

\begin{document}

\begin{abstract} We classify orthogonal actions of finite groups on Euclidean vector spaces for which the corresponding quotient space is a topological, homological or Lipschitz manifold, possibly with boundary. In particular, our results answer the question of when the underlying space of an orbifold is a manifold.
\end{abstract}

\maketitle

%
%
%

\section{Introduction}

The quotient of a finite-dimensional Euclidean vector space by a finite linear group action inherits different structures from the initial space, e.g. a topology, a metric and a piecewise linear structure. The question when such a quotient is a manifold has been asked in several ways \cite{Mikhailova,MR1118824,MR1935486,Petrunin,MR2883685}. Most notably, it was studied by Mikha\^ilova in the $70$s and $80$s. It led her to the investigation of finite groups generated by \emph{rotations}, i.e. by orthogonal transformations whose fixed-point subspace has codimension \emph{two}. We call such a group a \emph{rotation group}. Examples of rotation groups are orientation preserving subgroups of real reflection groups, and complex reflection groups considered as real groups. Large subclasses of rotation groups were classified in a series of papers by Mikha\^ilova \cite{Maerchik,MR608821,Mikhailova2}. In most of these cases she verified in \cite{Mikhailova} that the corresponding quotient is homeomorphic to the initial Euclidean space. In fact, she claimed to have proven an ``if and only if'' statement. However, apart from the fact that some rotation groups are not considered in \cite{Mikhailova} (cf. \cite{Lange1}), there is a counterexample in the topological category. The binary icosahedral group admits a faithful realization $P<\SOr_4$  and the quotient $S^3/P$ is Poincar\'e's homology sphere. We call a subgroup of the orthogonal group $\Or_n$ conjugate to $P<\SOr_4 \subset \Or_n$ a \emph{Poincar\'e group}. It follows from Cannon's double suspension theorem \cite{MR541330} that $\R^4/P\times \R^2$ is homeomorphic to $\R^6$  (see Lemma \ref{lem:poinc_group_double}), although the action of $P$ on $S^3$ is free and thus not generated by rotations. Without being aware of Mikha\^ilova's work, Davis conjectured in \cite{MR2883685} that, apart from the above mentioned examples, a Poincar\'e group is the only exceptional group with a quotient space homeomorphic to a Euclidean space.

In \cite{LaMik} new examples of rotation groups are described and a classification result is proven based on the earlier work of Mikha\^ilova and others. In fact, in \cite{LaMik} finite groups generated by reflections and rotations are classified. Partly using this classification, the author proves in \cite{Lange1} that the quotient of $\R^n$ by a finite subgroup $G$ of $\Or_n$ is a piecewise linear manifold with boundary if and only if $G$ is a \emph{reflection-rotation group}, i.e. it is generated by reflections and rotations. 
In this paper we show that Davis's conjecture holds true if one allows for general rotation groups. Moreover, we generalize the result as to include manifolds with boundary. More precisely, we prove the following statement.

\begin{thmA} Let $G < \Or_n$ be a finite subgroup. The quotient space $\R^n/G$ with the quotient topology is a topological manifold with boundary if and only if $G$ splits as a product
\begin{eqnarray*}
		G = G_{\mathrm{rr}} \times P_1 \times \cdots \times P_k
\end{eqnarray*}
of a reflection-rotation group $G_{\mathrm{rr}}<\Or(V_{\mathrm{rr}})$ and Poincar\'e groups $P_i<\SOr(V_{i})$, $i=1,\ldots,k$, that act in pairwise orthogonal spaces $\R^n=V_{\mathrm{rr}}\oplus V_{1} \oplus \cdots \oplus V_{k}$ with the following additional condition in the case $n\leq 5$: We must have $k=0$ if $n\leq 4$, or if $G_{\mathrm{rr}}$ contains a reflection. In this case $\R^n/G$ is either homeomorphic to $\R_{\geq 0}\times \R^{n-1}$ and $G$ contains a reflection, or $\R^n/G$ is homeomorphic to $\R^n$ and $G$ does not contain reflections.
\end{thmA}

In particular, the underlying space of an orbifold is a topological manifold if and only if all local groups (see Section \ref{sec:Riemannian_orbifolds}) have this special form, answering a question published by Davis \cite[p.~9]{MR2883685}.

The if direction of Theorem A is a corollary of the double suspension theorem (see Theorem \ref{thm:double_suspension_theorem}), the generalized Poincar\'e conjecture and the result of \cite{Lange1} (see Section \ref{sec:if_dir}). This paper is mainly concerned with the proof of the \emph{only if} statement. Roughly speaking, this proof is divided into four steps. The case of manifolds without boundary is treated in the first three steps. The general case is treated in the last step.

In the \emph{first step} we observe that, if $\R^n/G$ is a topological manifold for a finite subgroup $G<\Or_n$, then strata of $\R^n/G$ that are not contained in the closure of any higher dimensional singular stratum either have codimension two or codimension four, and the corresponding local (isotropy) groups are either cyclic groups or Poincar\'e groups (see Section \ref{sec:homology_criterion}). A key ingredient in this step is a theorem by Zassenhaus that characterizes representations of finite groups with certain properties related to the local homology groups of the corresponding quotient space (see Section \ref{sec:singular_poincare} and Proposition \ref{prp:characterization_minimal_subgroup_Top}). In this way we obtain a normal subgroup $N$ of $G$ generated by rotations and Poincar\'e groups that later turns out to be sufficiently large. In the \emph{second step} we use an elementary observation about spherical triangles and the specific geometric structure of the 600-cell, i.e. the orbit of one point under the action of the Poincar\'e group on $S^3$, to show that the rotation group and all Poincar\'e groups that generate $N$ act in pairwise orthogonal spaces (see Section \ref{sub:orthogonal_splitting}). In the \emph{third step} we show by induction on the dimension that $G/N \curvearrowright S^{n-1}/N$ is a free action on a space with the integral homology groups of a sphere. The algebraic information on $G/N$ obtained in this way suffices to identify $G/N$ as a trivial group (see Section \ref{sub:B_and_C_man}).

In the \emph{fourth step} we show that if $\R^n/G$ is a topological manifold with boundary, then $G$ contains a reflection and the double of $\R^n/G$ is homeomorphic to the quotient of $\R^n$ by the orientation preserving subgroup of $G$. This reduces the general case to the case of a manifold without boundary and an analysis of the action of the reflection. In the proof of this reduction step we apply general facts about Riemannian orbifolds and their coverings (see Section \ref{sec:Riemannian_orbifolds}). 

We state the following corollary from the proof of Theorem A since it answers a question by Petrunin \cite{Petrunin} and puts an observation by Dunbar in a broader context, see \cite[Final remarks]{MR1118824}. 

\begin{corB} Let $G < \Or_n$ be a finite subgroup. The quotient space $S^{n-1}/G$ with the quotient topology is homeomorphic to the $(n-1)$-sphere $S^{n-1}$ if any only if $G$ splits as a product
\begin{eqnarray*}
		G = G_{\rota} \times P_1 \times \cdots \times P_k
\end{eqnarray*}
of a rotation group $G_{\rota}<\Or(V_{\rota})$ and Poincar\'e groups $P_i<\SOr(V_{i})$, $i=1,\ldots,k$, that act in pairwise orthogonal spaces $\R^n=V_{\rota}\oplus V_{1} \oplus \cdots \oplus V_{k}$, with $k=0$ if $n\leq 5$.
\end{corB}

The appearance of Poincar\'e groups in Theorem A involves ``wild'' homeomorphisms \cite{MR0494113}. In case of manifolds without boundary it follows from a theorem of Siebenmann and Sullivan that this phenomenon cannot occur in the Lipschitz category (see Section \ref{sec:Lip_sec}). From this statement, the result of \cite{Lange1} and the fourth step in the proof of Theorem A we deduce the following statement.

\begin{corC}\label{thm:theorem_D} Let $G < \Or_n$ be a finite subgroup. The quotient space $\R^n/G$ with the quotient metric is a Lipschitz manifold with boundary if and only if $G$ is a reflection-rotation group. In this case $\R^n/G$ is bi-Lipschiz homeomorphic to $\R_{\geq 0}\times \R^{n-1}$ if $G$ contains a reflection, or $\R^n/G$ is bi-Lipschitz homeomorphic to $\R^n$ if $G$ does not contain reflections.
\end{corC}
In particular, (the underlying metric space of) a Riemannian orbifold is a Lipschitz manifold with boundary if and only if all local groups are reflection-rotation groups.
\newline
\newline
\emph{Acknowledgements.} I sincerely thank my advisor Alexander Lytchak for his encouragement and support. It is also a pleasure to thank Anton Petrunin for revealing the geometric meaning of Lemma \ref{lem:spherical_comparison} and Franz-Peter Heider for discussions and references on the group homology of $\mathrm{SL}_2(p)$. Moreover, I would like to thank the anonymous referees as well as Hansjörg Geiges and Ricardo Mendes for useful comments that helped to improve the exposition.

The author was partially supported by a `Kurzzeitstipendium für Doktoranden' by the German Academic Exchange Service (DAAD) and by the DFG funded project SFB/TRR 191. The support is gratefully acknowledged.

\section{Preliminaries}

\subsection{Riemannian orbifolds}
\label{sec:Riemannian_orbifolds}
For a definition of a smooth orbifold we refer the reader to e.g. \cite{MR1744486} or \cite{MR2883685}. A Riemannian orbifold can be defined as follows \cite{Lange2}.

\begin{dfn}\label{dfn:orbifold} A \emph{Riemannian orbifold} of dimension $n$ is a metric length space $\Orb$ such that for each point $x \in \Orb$ there exists an open neighborhood $U$ of $x$ in $\Orb$ and a Riemannian manifold $M$ of dimension $n$ together with a finite group $G$ acting by isometries on $M$ such that $U$ and $M/G$ are isometric.
\end{dfn}

Recall that a \emph{length space} is a metric space in which the distance between any pair of points can be realized as the infimum of the lengths of all rectifiable curves connecting these points \cite{MR1835418}. Here $M/G$ is endowed with the \emph{quotient metric}, i.e. the distance between two points of $M/G$ is defined as the distance between the respective orbits in $M$. Behind the above definition lies the fact that an isometric action of a finite group on a simply connected Riemannian manifold can be recovered from the corresponding metric quotient \cite{MR1935486,Lange2}. In particular, every Riemannian orbifold admits a canonical smooth orbifold structure and a compatible Riemannian structure that in turn induces the metric structure \cite{Lange2}. Conversely, every paracompact smooth orbifold admits a compatible Riemannian structure that turns it into a Riemannian orbifolds \cite[Ch.~III.1]{MR1744486}.

For a point $x$ on a Riemannian orbifold, the linearized isotropy group of a preimage of $x$ in a Riemannian manifold chart is uniquely determined up to conjugation. Its conjugacy class in $\Or_n$ is denoted as $G_x$ and it is called the \emph{local group} of $\Orb$ at $x$. The local group $G_x$ of a point $x \in \Orb$ determines the topology and geometry of $\Orb$ in a neighborhood of $x$ via the exponential map in the following way (see e.g. \cite[Lem.~100]{Lange_thesis}).

\begin{lem} \label{lem:lipschitz_orbifold}
For an $n$-dimensional Riemannian orbifold $\Orb$ and a point $x \in \Orb$ there exists a neighborhood of $x$ in $\Orb$ that is locally bi-Lipschitz homeomorphic to $\R^n/G_x$.
\end{lem}

Suppose that $\Orb$ is a global quotient of a manifold $M$ by a smooth action of a finite group $G$. Then the images of the sets
\[
		M_{(H)}=\{x \in M | G_x \text{ is conjugate to } H \text{ in }G\} 
\]
under the natural projection $M \To \Orb$, where $H<G$ is a subgroup, define a stratification of $M$ by manifolds \cite[Prop.~1.12]{MR2883685}. The stratum corresponding to the trivial subgroup is called \emph{regular}. All other strata are called \emph{singular}. 

For a general Riemannian orbifold $\Orb$ there exists a rougher stratification by manifolds where the \emph{codimension $i$ stratum} $\Sigma_i=\Sigma_i \Orb$ consists of all points $x \in \Orb$ with $\mathrm{codim} \mathrm{Fix}(G_x)=i$. In accordance with the terminology above, points in $\Sigma_0 \Orb$ are called regular whereas all other points are called singular.

The remaining discussion is only relevant for the proof of Theorem A in case of manifolds with boundary (see Section \ref{sub:boundary_case}). In this case we are concerned with the closure of the codimension $1$ stratum of a Riemannian orbifolds $\Orb$, which can be characterized as follows (see \cite[Lem.~2.5]{Lange2}).

\begin{lem}\label{lem:charact_codim_1_stratum}
A point $x \in \Orb$ belongs to the closure of the codimension $1$ stratum of $\Orb$, if and only if its local group $G_x$ contains a reflection.
\end{lem}
We refer to the closure of $\Sigma_1 \Orb$ as the \emph{boundary} of $\Orb$ and denote it as $\partial \Orb$ (since it coincides with the boundary of $\Orb$ in the sense of Alexandrov geometry by Lemma \ref{lem:charact_codim_1_stratum}, cf. \cite{MR1185284}).

The topological double $2_{\partial \Orb} \Orb$ can we endowed with a natural metric, which can be described as the unique maximal metric that is majorized by the metrics on the two copies of $\Orb$ in $2_{\partial \Orb}  \Orb$ \cite[3.1.24]{MR1835418} (or, alternatively, as a \emph{gluing metric} \cite[3.1.12, 3.1.27]{MR1835418}, cf. \cite{Lange2} and Section \ref{sec:Lip_sec}). With respect to this metric the two copies of $\Orb$ are isometrically embedded in $2_{\partial \Orb} \Orb$ (cf. \cite[Lem.~5.3.]{Lange2}) and $2_{\partial \Orb} \Orb$ is a length space \cite[3.1.24]{MR1835418}. We call $2_{\partial \Orb} \Orb$ the \emph{metric double} of $\Orb$.

A \emph{covering} $p:\Orb' \To \Orb$ of complete Riemannian orbifolds $\Orb'$ and $\Orb$ of the same dimension can be defined as a submetry, i.e. a map that maps the closed ball $B_r(x)$ onto the closed ball $B_r(p(x))$ for any $x \in \Orb'$ and any $r>0$. Locally such a map is of the form $M/H \To M/G$ for a subgroup $H$ of $G$ as in Definition \ref{dfn:orbifold} (see \cite{Lange2}). In fact, Thurston's original definition is based on this property \cite{Thurston}. The following statement is known and proven in \cite{Lange2}. It will be applied in the proof of Theorem A (see Section \ref{sub:boundary_case}).

\begin{prp}\label{prp:orbifold_double} Let $\Orb$ be a Riemannian orbifold with nonempty boundary. Then the metric double of $\Orb$ along its boundary is a Riemannian orbifold with empty boundary and the natural projection to $\Orb$ is a covering of Riemannian orbifolds.
\end{prp}

\subsection{Piecewise linear manifolds and admissible triangulations}
\label{sub:ad_triangulations}

A \emph{piecewise linear triangulation} of a subset $U\subseteq\R^n$ is a simplicial complex $K$ in $U$ whose underlying space is all of $U$. Let $G<\Or_n$ be a finite subgroup. We say that $G$ acts \emph{simplicially} on a piecewise linear triangulation $K$ of $\R^n$ if it maps simplices of $K$ linearly onto simplices of $K$. We call such a triangulation \emph{admissible (for the action of $G$ on $\R^n$)}, if the action of $G$ on $K$ is in addition regular in the sense of \cite[Ch.~III, Def.~1.2]{MR0413144}, and if $K$ contains the origin as a vertex. Regularity ensures that the quotient $|K|/G$ is in a natural way a simplicial complex which we denote by $K/G$ \cite{MR0413144}. Regularity can always be assumed by passing to the second barycentric subdivision. In particular, regularity implies that if some $g\in G$ fixes a point in the interior of a simplex of $K$, then $g$ fixes this simplex pointwise.

Admissible triangulations always exist \cite[Lem.~3.2]{Lange1}. In particular, all quotients $\R^n/G$ and $S^{n-1}/G$, $G<\Or_n$ finite, that we are working with are homeomorphic to locally finite, finite-dimensional simplical complexes. Moreover, the simplicial complex $K/G$ provides the topological quotient $\R^n/G$ with a \emph{quotient piecewise linear structure} \cite{Lange1}. The quotient $\R^n/G$ being piecewise linear homeomorphic to a subset $U\subseteq \R^n$ means that there exists a simplicial subdivision of $K/G$ that is simplicially isomorphic to a piecewise linear triangulation of $U$; cf. \cite{Lange1}. The following result is proven in \cite{Lange1}.

\begin{thm}\label{thm:theorem_Lange} Let $G < \Or_n$ be a finite subgroup. The quotient space $\R^n/G$ with the quotient piecewise linear structure is a piecewise linear manifold if and only if $G$ is a reflection-rotation group. In this case the quotient space $\R^n/G$ is piecewise linear homeomorphic to $\R_{\geq 0}\times \R^{n-1}$ and $G$ contains a reflection, or $\R^n/G$ is piecewise linear homeomorphic to $\R^n$ and $G$ does not contain reflections. In particular, in these cases $S^{n-1}/G$ is homeomorphic to an $(n-1)$-ball or an $(n-1)$-sphere, respectively.
\end{thm}

\subsection{Poincar\'e groups}
\label{subsec:the_binary_icosahedral_group}
In the following we identify $\R^4$ and $\Co^2$ with the algebra of quaternions $\mathbb{H}$ via 
\[
	 (x_1,y_1,x_2,y_2) \mathrel{\widehat{=}} (x_1+iy_1,x_2+i y_2) \mathrel{\widehat{=}} x_1+iy_1+jx_2+k y_2 .
\]
Moreover, we identify the special unitary group $\SU_2$ with the unit quaternions in $\mathbb{H}$ via
\[
	 \left(
		  \begin{array}{cc}
		    z_1 & -\overline{z_2} \\
		    z_2 & \overline{z_1} \\
		  \end{array}
		\right)=\left(
		  \begin{array}{cc}
		    x_1+iy_1 & -x_2+iy_2 \\
		    x_2+iy_2 & x_1-iy_1 \\
		  \end{array}
		\right)
  \mathrel{\widehat{=}} x_1+iy_1+jx_2+k y_2.
\]
This identification defines an isomorphism of Lie groups. In particular, under these identifications the natural action of $\SU_2$ on $\Co^2$ translates to the action of unit quaternions on $\mathbb{H}$ by left multiplication.

The subspace $\mathbb{I}=\R i\oplus\R j \oplus \R k$ of $\mathbb{H}$ of purely imaginary quaternions is invariant under conjugation by $\SU_2$. The map
\[
		 \begin{array}{cccl}
		 \psi : &\SU_2 				& \To  	&	\SOr_3=\SOr(\mathbb{I}) \\
		           		& g 	& \MTo  & \varphi(g): q \MTo gqg^{-1}
		 \end{array}
\]
is a twofold covering map of Lie groups. The preimage of the orientation preserving symmetry group of a centered icosahedron in $\R^3$ under the map $\psi$ is called the \emph{binary icosahedral group} and we denote it by $\ico$. It is isomorphic to $\SL_2(5)$, the special linear group of a $2$-dimensional vector space over the finite field of order five, and a \emph{perfect} group \cite[p.~196]{MR928600}, i.e. it coincides with its commutator subgroup. With a specific choice of coordinates the binary icosahedral group is given by the union of the $24$ Hurwitz units
\[
		\{\pm 1,\pm i,\pm j,\pm k, \frac{1}{2}(\pm 1\pm i\pm j\pm k)\}
\]
together with all $96$ unit quaternions obtained from $\frac{1}{2}(\pm i\pm\tau^{-1} j\pm \tau k)$ by an even permutation of coordinates, where $\tau:=(1+\sqrt{5})/2$ is the golden ratio \cite{MR0169108}.

The natural action of $\SU_2$ on $\Co^2\mathrel{\widehat{=}}\R^4$ preserves the standard inner product of $\R^4$, and so $\SU_2$ embeds into $\SOr_4$. We refer to the image of the binary icosahedral group in $\SOr_4$ as a Poincar\'e group. In fact, we define a \emph{Poincar\'e group} as a subgroup of $P < \Or_n$ that is conjugate in $\Or_n$ to the image of the binary icosahedral group in $\SOr_4\subset \Or_n$. 

Since $\ico<\SU_2$ acts by left multiplications on $\SU_2$ the action of a Poincar\'e group $P<\SOr_4$ on the unit sphere $S^3 \subset \R^4$ is free. Hence, the quotient $S^3/P$ is a manifold. Moreover, since $P$ is a perfect group, we have $H_1(S^3/P;\Z)=0$ (cf. Section \ref{sec:algebraic_singular_poincare}). Thus, Poincar\'e duality and the universal coefficient theorem imply that $S^3/P$ is an (integral) \emph{homology sphere}, i.e. a topological manifold with the integral homology groups of a sphere. In fact, it is \emph{Poincar\'e's homology sphere} \cite{MR0537730}.

\subsection{Homology manifolds}\label{sub:homology_manifold}
The difficulty in proving the if direction of Theorem A compared to the if direction of Theorem \ref{thm:theorem_Lange} and Corollary C (see Section \ref{sec:if_cor_c}) is that the quotient $S^{n-1}/G$, $G< \Or_n$, does not need to be a topological manifold if $\R^n/G$ is so. Now we explain this in more detail. In the following the symbol ``$\cong$'' between two topological spaces indicates that the spaces are homeomorphic.

The \emph{join $X*Y$} of two topological spaces $X$ and $Y$ is defined as the quotient of $X \times Y \times [0,1]$ obtained by collapsing $\{x\}\times Y \times \{0\}$ to a point for any $x \in X$, and collapsing $X\times \{y\} \times \{1\}$ to a point for any $y \in Y$ (always with respect to the quotient topology). The \emph{suspension} $\Sigma X$ of $X$ can be defined as the join $\{0,1\} * X$. The \emph{double suspension} $\Sigma^2 X:=\Sigma(\Sigma(X))$ is homeomorphic to the join $S^1 * X$.

The following result bases on work by Edwards and Cannon and was first obtained in full generality by Cannon \cite{MR541330}.
\begin{thm}[Cannon, Edwards]
The double suspension $\Sigma^2 X$ of any homology $n$-sphere $X$ is a topological $(n+2)$-sphere.
\label{thm:double_suspension_theorem}
\end{thm}

In particular, the double suspension $\Sigma^2(S^3/P)$ of Poincar\'e's homology sphere is a topological $5$-sphere. However, $\Sigma(S^3/P)\cong S^4/P$, $P<\SOr_4<\SOr_5$, is not a topological manifold since complements of the suspension points in small neighborhoods are not simply connected. The \emph{open cone} of a topological space $X$ is defined as the quotient $CX=(X \times [0,\infty)) / (X \times \{0\})$.
The following lemma gives an example of the phenomenon mentioned at the beginning of this section.
\begin{lem}
Let $P < \SOr_4$ be a Poincar\'e group. Then there are homeomorphisms $\R \times \R^4/P \cong C(S^4/P) \cong \R^5$.
\label{lem:poinc_group_double}
\end{lem}
\begin{proof}
Recall that the suspension $\Sigma X$ of a topological space $X$ is homeomorphic to the join $X * S^0$. Moreover, the open cone $CX$ is homeomorphic to the suspension $\Sigma X$ with one suspension point removed. The lemma is a consequence of the following chain of homeomorphisms
\begin{alignat*}{1}
		 \R \times \R^4/P &\cong \R \times (CS^3)/P \cong  CS^0 \times C(S^3/P) \cong C(S^0 * S^3/P )\\
								&  \cong C\Sigma(S^3/P) \cong \Sigma^2(S^3/P)- \{*\} \cong S^5 -\{*\} \cong \R^5.
\end{alignat*}
Here we have used that $(CS^3)/P \cong  C(S^3/P)$, the property $C(X_1 * X_2)\cong CX_1 \times CX_2$ of the join \cite[Prop.~I.5.15]{MR1744486} and the double suspension theorem.
\end{proof}

In order to avoid this problem, we work with \emph{homology manifolds}, instead of topological manifolds. Homology manifolds are generalizations of topological manifolds and, for instance, arise in the problem of recognizing the latter \cite{MR0494113}. More precisely, we work with the following definition like in \cite{MR1402473} where integer coefficients are understood.

\begin{dfn}\label{dfn:hom_mfd} Let $X$ be a connected topological space homeomorphic to a locally finite, finite-dimensional simplicial complex. Suppose that for all $x\in X$ we have $H_i(X, X-\{x\})\cong H_i(B^n, B^n-\{p\})$ for some point $p\in B^n$ in the unit-ball in $\R^n$, i.e.
\[
	H_i(X, X-\{x\})=\begin{cases}
  0  &\text{for }i\neq n\\
  0 \text{ or } \Z &\text{for }i= n.\\
\end{cases} 
\]
Then we say that $X$ is a \emph{homology $n$-manifold with boundary}. The \emph{boundary $\partial X$} of $X$ is defined to be
\[
		\partial X := \{x\in X \mid H_n(X, X-\{x\})=0\}.
\]
If the boundary of $X$ is empty, then we call $X$ a \emph{homology $n$-manifold}.
\end{dfn}

\begin{rem} There are definitions of homology manifolds with weaker assumptions on $X$ \cite{MR1886687,MR1019276}. For instance, in \cite{MR1019276} the space $X$ is only demanded to be locally compact and of finite cohomological dimension. Since all spaces we are dealing with are homeomorphic to finite-dimensional locally-compact simplicial complexes by Section \ref{sub:ad_triangulations}, these variations do not make a difference for us.
\end{rem}

If the space $X$ in Definition \ref{dfn:hom_mfd} is a topological manifold with boundary, then it is also a homological manifold with boundary. Moreover, its boundary as a topological manifold coincides with its boundary as a homological manifold.

For a topological space $X$ and a subspace $Y \subset X$ we define the \emph{double of $X$ along $Y$} to be
\[
	 2_Y X = X \times \{0,1\} / \sim \text{ where } (y,0)\sim (y,1) \text{ for all }y\in Y.
\]
We simply denote it by $2X$ if the meaning of the subspace is clear.

Each connected component of the boundary of a homology $(n+1)$-manifold with boundary $X$ is a homology $n$-manifold and closed in $X$, and the double $2X:=2_{\partial X} X$ is a homology $(n+1)$-manifold \cite[Ch.~5]{MR1402473}.

An example of a homology manifold which is not a topological manifold is given by the quotient space $\R^4/P\cong C (S^3/P)$, $P< \SOr_4$ being a Poincar\'e group. Indeed, it is not a topological manifold since complements of the coset of the origin in small neighborhoods are not simply connected, and it is a homology manifold by the following lemma. Its proof consists of standard computations in algebraic topology using the K\"{u}nneth formula and long exact sequences for pairs (see e.g. \cite[Lem.~101]{Lange_thesis}). Moreover, the lemma shows that the initially mentioned problem does not occur for homology manifolds.
\begin{lem} \label{lem:characterization_of_homology_manifolds}
Let $X$ and $Y$ be topological spaces homeomorphic to locally finite, finite-dimensional simplicial complexes. Then the following statements hold for integers $n\geq0$.
\begin{compactenum}
	\item $X$ and $Y$ are homology manifolds if and only if $X\times Y$ is a homology manifold.
 	\item $CX$ is a homology $(n+1)$-manifold if and only if $X$ is a compact homology $n$-manifold and $H_*(X)=H_*(S^n)$.
	\item If $X$ is a homology manifold, then $X\times Y$ is a homology manifold with boundary, if and only if $Y$ is a homology manifold with boundary. In this case we have $\partial (X\times Y) = X \times \partial Y$.
	\item $CX$ is a homology $(n+2)$-manifold with nonempty boundary, if and only if $X$ is a compact homology $(n+1)$-manifold with nonempty boundary and $H_*(X)=H_*(\{*\})$, $H_*(\partial X)=H_*(S^n)$. In this case we have $\partial (CX)=C(\partial X)$.
\end{compactenum}
\end{lem}

\section{The singular role of the Poincar\'e groups}
\label{sec:singular_poincare}

In Section \ref{sub:homology_manifold} we have seen that Poincar\'e groups occur as examples in our Theorem A. In this section we explain the algebraic origin of their singular role. Recall from Section \ref{subsec:the_binary_icosahedral_group} that a Poincar\'e group is abstractly isomorphic to the special linear group $\SL_2(5)$. The following statement is proven in \cite{MR0310895}.

\begin{thm}[Sjerve] If $M^n$ is a homology sphere which is covered by $S^n$ and if $n\geq 3$, then either $\pi_1(M)\cong \{1\}$, or $n=3$ and $\pi_1(M)\cong \SL_2(5)$.
\label{thm:Sjerve}
\end{thm}

This theorem generalizes a result by Zassenhaus. Namely, in \cite{MR3069653,MR3069657} (cf. \cite{MR928600}) he proves that the only non-trivial finite perfect group admitting an irreducible complex fixed-point free representation $\rho$ is $\SL_2(5)$ and that the degree of $\rho$ is two. A complete proof of this statement can be found in Wolf's book \cite[Thm.~6.2.1]{MR928600}. These results are pivotal for our Theorem A. We will also need the following statement which is very much related to Theorem \ref{thm:Sjerve}.

\begin{prp}\label{prp:char_sl_homology} Let $G<\Or_{n+1}$ be a finite subgroup and let $N\triangleleft G$ be a normal subgroup such that both $\R^{n+1}/G$ and $\R^{n+1}/N$ are homology manifolds and such that the action $G/N \curvearrowright S^{n}/N$ is free. Suppose that $n\geq 3$. Then $G/N$ is a perfect group. Moreover, either $G=N$, or $n=3$ and $G/N \cong \SL_2(5)$.
\end{prp}

In the following Section we explain the ingredients of the proofs of Proposition \ref{prp:char_sl_homology} and Theorem \ref{thm:Sjerve}, and reduce the proof of Proposition \ref{prp:char_sl_homology} to the proof of Theorem \ref{thm:Sjerve}.

\subsection{Proof of Proposition \ref{prp:char_sl_homology}}
\label{sec:algebraic_singular_poincare}

As a first step, we need the following lemma. 

\begin{lem}\label{lem:ac_hom_quo} Under the assumptions on $G$ and $N$ in Proposition \ref{prp:char_sl_homology} we have $G < \SOr_{n+1}$ and the induced action of $G/N$ on $H_n(S^n/N)\cong \Z$ is trivial. In particular, the conclusion holds for a finite subgroup $G<\Or_{n+1}$ and trivial $N$ if $G$ acts freely on $S^n$ and $\R^{n+1}/G$ is a homology manifold.
\end{lem}
\begin{proof}
By Lemma \ref{lem:characterization_of_homology_manifolds}, $(ii)$, we have $H_{*}(S^{n}/G)=H_{*}(S^{n}/N)=H_{*}(S^n)$. There are transfer maps $\tau_*:H_{n}(S^{n}/G) \To H_{n}(S^n)$ and $t_*: H_n(S^{n}/N) \To H_{n}(S^{n})$ for $\pi_{*}: H_{n}(S^{n}) \To H_{n}(S^{n}/G)$ and $p_*: H_n(S^{n}) \To H_{n}(S^{n}/N)$ with $\pi_*\circ \tau_* = |G|\cdot \mathrm{id}$ and $p_*\circ t_* = |N|\cdot \mathrm{id}$ \cite[Ch.~3.G]{MR1867354}. In particular, the maps $\pi_{*}$ and $p_*$ are nontrivial. For every $g\in G$ the diagrams 
\[
	\begin{xy}
		\xymatrix
		{
		   H_{n}(S^{n}) \ar[rr]^{g_*=\text{deg}(g)=\text{det}(g)} \ar[dr]_{\pi_*} & & H_{n}(S^{n}) \ar[dl]^{\pi_*} &  H_{n}(S^{n}) \ar[rr]^{g_*} \ar[d]_{p_*} & & H_{n}(S^{n}) \ar[d]_{p_*} \\
		   & H_{n}(S^{n}/G)  &  & H_{n}(S^{n}/N) \ar[rr]^{\overline{g}_*} & & H_{n}(S^{n}/N)  
		}
	\end{xy}
\]
commute. All the spaces in these diagrams are isomorphic to $\Z$, all the maps correspond to multiplication by a nontrivial integer and the maps ${g_*}$ and ${\overline{g}_*}$ are invertible. Therefore ${g_*}$ and ${\overline{g}_*}$ are the identity. It follows that $G < \SOr_{n+1}$ and that $G/N$ acts trivially on $H_{*}(S^{n}/N)$ as claimed.
\end{proof}

For the further argument we first recall some facts about the cohomology of finite groups from \cite{Adem,MR1324339}. Given a finite group $G$ one can construct a contractible CW-complex $EG$ on which $G$ acts freely and whose cells are permuted by $G$. The quotient $BG=EG/G$ is then a CW-complex as well. It is a $K(G,1)$ Eilenberg-MacLane space, meaning that its fundamental group is isomorphic to $G$ and all its higher homotopy groups vanish. The homology and cohomology groups of $G$ with coefficients in a $G$-module $M$ are defined as the homology and cohomology groups of $BG$ with (local) coefficients in $M$, that is
\[
			H_*(G;M):=H_*(BG;M) \text{ and } H^*(G;M):=H^*(BG;M).
\]
Note that the homotopy type of a $K(G,1)$ space, and hence $H_*(G;M)$ and $H^*(G;M)$, is determined by the group $G$ \cite[Thm.~1B.8]{MR1867354}. The (co)homology groups of $G$ can be computed from a projective resolution of $\Z$ over the group ring $\Z G$. Namely, if $F$ is such a resolution and $M$ is some $G$-module, then $H_*(G,M)=H_*(F \otimes_{\Z G} M)$ and $H^*(G,M)=H^*(\mathrm{Hom}_{\Z G}(F,M))$ holds \cite[III.1]{MR1324339}. The well-known identifications $\pi_1(BG)\cong G$ and $H_1(G;\Z)\cong \pi_1(BG)_{\mathrm{ab}}\cong G/[G,G]$ \cite[Thm.~2A.1]{MR1867354} show that  $H_1(G;\Z)$ vanishes if and only if the group $G$ is perfect.

For $d\geq 1$ a finite group $G$ is said to have \emph{$d$-periodic cohomology} if one of the following conditions is satisfied \cite[Thm.~VI.9.1]{MR1324339}, \cite[Prop.~XII 11.1]{MR0077480}.
\begin{thm}[Artin-Tate]\label{thm:Artin-Tate} For $d\geq 1$ the following conditions are equivalent.
\begin{compactenum}
	\item There exists some $n\geq 1$ such that $H^n(G;M)$ and $H^{n+d}(G;M)$ are isomorphic for all $G$-modules $M$.
	\item $H^d(G;\Z) \cong \Z/|G|\Z$.
\end{compactenum}
\label{thm:char_peri_coh}
\end{thm}
If the conditions in the theorem are satisfied, then a generator $u$ of $H^d(G;\Z)$ induces isomorphisms
\[
			u \cup \cdot :H^n(G;M) \To H^{n+d}(G;M)
\]
for all $G$-modules $M$. (In fact, this element $u$ is invertible in the so-called \emph{Tate cohomology ring} of $G$, and the existence of such a $u$ is usually taken as the definition of $d$-periodic cohomology \cite{MR1324339,MR0077480}).

In the following the (trivial) coefficient ring $\Z$ is understood if not specified otherwise. The next proposition can be applied in the situation of Theorem \ref{thm:Sjerve} to the action of the deck transformation group on $S^n$, and in the situation of Proposition \ref{prp:char_sl_homology} to the action of $G/N$ on $S^n/N$ by Lemma \ref{lem:characterization_of_homology_manifolds}, $(ii)$.
\begin{prp} Let $X$ be a compact topological space homeomorphic to an $n$-dimensional simplicial complex, and let $G \curvearrowright X$ be a free action of a finite group $G$. Assume that $H_*(X)=H_*(S^n)$ and that the action of $G$ on $H_n(X)$ is trivial. Then we have
\[
		H^i(G)=H^i(X/G) \text{, } H_i(G)=H_i(X/G), \ \text{for } 0\leq i < n
\]
and $G$ has $(n+1)$-periodic cohomology.
\label{prp:connection-group-topological-homology}
\end{prp}
\begin{proof} We sketch an argument from \cite{MR1324339} which proves the claim in the case in which $X$ is a simplicial complex on which $G$ acts simplicially and regularly (cf. Section \ref{sub:ad_triangulations}). We give a reference for another argument in the general case. Note however that, in the situation in which we apply the present proposition, namely to the action of $G/N$ on $S^n/N$ for some finite subgroups $N\triangleleft G < \Or_{n+1}$, the additional assumptions can be easily arranged (cf. Section \ref{sub:ad_triangulations}).

Let $C_*=C_*(X)$ be the chain complex of the simplicial complex $X$. We have an exact sequence of $G$-modules
\[
		0 \To \Z \To C_n \To \dots \To C_1 \To C_0 \To \Z \xrightarrow[]{\varepsilon} 0,
\]
where each $C_i$ is free, the $G$-action on the two copies of $\Z$ is trivial, the map $\Z \To C_n$ sends $1\in \Z$ to a generator of the cycle subgroup of $C_n$ and $\varepsilon$ is the augmentation map. Splicing together an infinite number of copies of this sequence yields a periodic free resolution of $\Z$ over $\Z G$, see \cite[I.6]{MR1324339}. In particular, $G$ has periodic cohomology by Theorem \ref{thm:Artin-Tate}, $(i)$, and the claim on the (co)homology groups of $X/G$ follows because there are isomorphisms of chain complexes $C_*(X/G)\cong\Z\otimes_{\Z G} C_*(X)$ and $C^*(X/G)\cong\mathrm{Hom}_{\Z G}(C_*(X),\Z))$ \cite[Prop.~II.2.4]{MR1324339}.

In the general case the proposition can be proven as in \cite[Ch.~XVI, \S 9, Appl.~4]{MR0077480} by an application of the Leray-Serre spectral sequence to the fibration $X\To X \times_G EG \To BG$ whose total space is homotopy equivalent to $X/G$, see also \cite[Ch.~9.7]{MR1841974}.
\end{proof}

Examples of finite groups with periodic cohomology are the finite subgroups of $\SU_2$ since they act freely on $S^3$. The binary dihedral groups among them have order $4n$ and can be presented as $\left\langle x,y | x^{2n}=1, y^2=x^n, y^{-1} x y =x^{-1}\right\rangle$. For $n$ being a power of $2$ they are also referred to as \emph{generalized quaternion groups}. The following characterization is proven in \cite[Thm.~XII 11.6]{MR0077480}. The first equivalence is due to Artin and Tate \cite{MR0223335}. The last follows from work of Burnside, Hall, et. al., cf. \cite[Thm.~VI.9.3]{MR1324339}.

\begin{thm}[Artin-Tate, Burnside, Hall, et. al.] Let $G$ be a finite group. The following conditions are equivalent.
\begin{compactenum}
	\item $G$ has periodic cohomology.
	\item Every abelian subgroup of $G$ is cyclic.
	\item The Sylow subgroups of $G$ are either cyclic or generalized quaternion groups.
\end{compactenum}
\label{thm:Brown_characterization_periodic_cohomology}
\end{thm}

If all Sylow subgroups are cyclic then $G$ is solvable by a theorem of Burnside \cite[Ch.~IX.128]{Burnside}. Finite groups with periodic cohomology have been classified by Suzuki and Zassenhaus \cite{Suzuki,MR3069653}. For any prime $p$, the group $\SL_2(p)$ of $2\times 2$ matrices of determinant 1 over the prime field of order $p$ has periodic cohomology \cite[p.~157]{MR1324339}. The following statement is an easy consequence of the classification as observed in \cite[Cor.~2.6]{MR0310895} and \cite[Prop.~105]{Lange_thesis}.

\begin{cor}
Let $G$ be a nontrivial finite perfect group. Then $G$ has periodic cohomology if and only if it is isomorphic to $\SL_2(p)$ for some prime $p>3$.
\label{lem:perfect_periodic_cohomology}
\end{cor}

Let us explain how the proof of Theorem \ref{thm:Sjerve} and Proposition \ref{prp:char_sl_homology} can be completed in the terminology of Proposition \ref{prp:char_sl_homology}.

\begin{proof}[Proof of Proposition \ref{prp:char_sl_homology}]
We can assume that the quotient group $G/N$ is nontrivial. By the Lefschetz fixed point theorem applied to the free action of $G/N$ on $S^n/N$ the dimension $n$ must be odd. By Lemma \ref{lem:characterization_of_homology_manifolds}, $(ii)$, Lemma \ref{lem:ac_hom_quo} and Proposition \ref{prp:connection-group-topological-homology} the group $G/N$ has periodic cohomology and trivial (co)homology in dimension $0< i < n$. In particular, it is perfect and thus isomorphic to $\SL_2(p)$ for some prime $p>3$ by Corollary \ref{lem:perfect_periodic_cohomology}. 

On the other hand, one can either deduce from Schur's work \cite{Schur} that $H_3(\SL_2(p);\Z) \neq 0$ for any prime $p$ as in \cite[Sec.~4.2.3]{Lange_thesis}, or based on \cite{MR0122856} that $H^4(\SL_2(p);\Z) \neq 0$ for any prime $p\geq5$ as in \cite{MR0310895}. Because $n$ is odd and $G/N$ is nontrivial, it follows in both cases that $n=3$. In particular, the acting group has $4$-periodic cohomology.

This is actually all we need (cf. Lemma \ref{lem:normal_subgroup_generated_by_ps_in_dim4}). To conclude that, in the case $G\neq N$, the acting group is isomorphic to $\SL_2(5)$ it is shown in \cite{MR0310895}, based on \cite{MR0122856}, that the minimal cohomological period of $\SL_2(p)$ is $\mathrm{lcm}(4,p-1)$.
\end{proof}

\subsection{Singular role of Poincar\'e groups as linear groups} In the preceding section we gave a criterion to identify Poincar\'e groups as abstract groups. In this section we give criteria to identify them as linear groups. Note that, for a finite subgroup $G<\Or_n$, the subgroup $G_{\rota}$ of $G$ generated by all rotations in $G$ is normal in $G$ since the set of rotations in $G$ is closed under conjugation by $G$. Also note that a Poincar\'e group $P<\SOr_4$ does not contain any rotations since its action on $S^3$ is free (cf. Section \ref{subsec:the_binary_icosahedral_group}). We prove the following two lemmas.
\begin{lem}
Let $G < \SOr_4$ be a finite subgroup and let $G_{\rota} \triangleleft G$ be the normal subgroup generated by the rotations contained in $G$. If $G / G_{\rota}$ is a nontrivial perfect group, then $G$ is a Poincar\'e group $P<\SOr_4$. In particular, the conclusion holds if $G/G_{\rota}$ is isomorphic to $\SL_2(5)$.
\label{lem:normal_subgroup_generated_by_ps_in_dim4}
\end{lem}
\begin{lem}\label{lem:char_p_in_SO(4)}
Any finite subgroup $G < \SOr_4$ isomorphic to $\SL_2(5)$ is a Poincar\'e group.
\end{lem}
The proofs of these lemmas can be deduced from the classification of finite subgroups of $\SOr_4$ \cite{MR0169108}. We first remind the reader of this classification.

Recall from Section \ref{subsec:the_binary_icosahedral_group} how we identify $\R^4$ with $\Co^2$ and with the algebra of quaternions $\mathbb{H}$, and $\SU_2$ with the unit quaternions in $\mathbb{H}$. There are twofold covering maps of Lie groups $\psi:\SU_2 \To \SOr_3$ and $\varphi: \SU_2 \times \SU_2 \To \SOr_4$. The first map has been described in Section \ref{subsec:the_binary_icosahedral_group}. The second map is explicitly given by
\[
		 \begin{array}{cccl}
		 \varphi : &\SU_2 \times \SU_2 				& \To  	&	\SOr_4=\SOr(\mathbb{H}) \\
		           		& (l,r) 	& \MTo  & \varphi((l,r)): q \MTo lqr^{-1}
		 \end{array}
\]
and has kernel $\{\pm(1,1)\}$.

Using these covering maps, the finite subgroups of $\SOr_4$ can be described based on the knowledge of the finite subgroups of $\SOr_3$. These are cyclic groups $\mathfrak{C}_n$ of order $n$, dihedral groups $\mathfrak{D}_n$ of order $2n$ and the symmetry groups of a tetrahedron, an octahedron and an icosahedron. The latter three groups are isomorphic to the alternating group $\alt_4$, the symmetric group $\sym_4$ and the alternating group $\alt_5$, respectively. The finite subgroups of $\SU_2$ are cyclic groups $\mathbf{C}_n$ of order $n$, binary dihedral groups $\mathbf{D}_n$ of order $4n$ and binary tetrahedral, octahedral and icosahedral groups denoted by $\mathbf{T}$, $\mathbf{O}$ and $\mathbf{I}$, respectively. Except for $\mathbf{C}_n$ with odd $n$, these are two-to-one preimages of respective subgroups of $\SOr_3$. In other words, the groups $\mathbf{C}_n$ with odd $n$ are the only subgroups of $\SU_2$ that do not contain the kernel of $\psi$.

The finite subgroups of $\SOr_4$ can be described in terms of the finite subgroups of $\SU_2$ and the covering map $\varphi$ as follows \cite[p.~54]{MR0169108}, \cite[Prop.~14]{LaMik}.

\begin{prp} \label{prp:so_subgroups}
Let $G \subgr  \SOr_4$ be a finite subgroup. Then there exist finite subgroups $\mathbf{L},\mathbf{R} \subgr  \SU_2$ with $-1\in \mathbf{L},\mathbf{R}$ and normal subgroups $\mathbf{L}_K \triangleleft \mathbf{L}$ and $\mathbf{R}_K \triangleleft \mathbf{R}$ such that $\mathbf{L}/\mathbf{L}_K$ and $\mathbf{R}/\mathbf{R}_K$ are isomorphic via an isomorphism $\phi: \mathbf{L}/\mathbf{L}_K \To \mathbf{R}/\mathbf{R}_K$ for which
\[
			G = \varphi(\{(l,r)\in \mathbf{L} \times \mathbf{R} | \phi(\pi_L(l))=\pi_R(r)\})
\]
holds. Here $\pi_L: \mathbf{L} \To \mathbf{L}/\mathbf{L}_K$ and $\pi_R: \mathbf{R} \To \mathbf{R}/\mathbf{R}_K$ are the natural projections. In this case we write $G=(\mathbf{L},\mathbf{L}_K;\mathbf{R},\mathbf{R}_K)_{\phi}$. Conversely, a set of data $(\mathbf{L},\mathbf{L}_K;\mathbf{R},\mathbf{R}_K)_{\phi}$ with the above properties defines a finite subgroup $G$ of $\SOr_4$ by the equation above.
\end{prp}

Given a group $G=(\mathbf{L},\mathbf{L}_K;\mathbf{R},\mathbf{R}_K)_{\phi}$ as in the proposition, the group
\[
			\Gamma:=[\mathbf{L},\mathbf{L}_K;\mathbf{R},\mathbf{R}_K]_{\phi}:=\{(l,r)\in \mathbf{L} \times \mathbf{R} | \phi([l])=[r]\} < \mathbf{L} \times \mathbf{R}
\]
is mapped two-to-one onto $G$ via $\varphi$ and we have
\[
			|G|=\frac{1}{2}|\mathbf{L}||\mathbf{R}_K|=\frac{1}{2}|\mathbf{R}||\mathbf{L}_K|.
\]
For a normal subgroup $\tilde{G}=(\tilde{\mathbf{L}},\tilde{\mathbf{L}}_K;\tilde{\mathbf{R}},\tilde{\mathbf{R}}_K)_{\phi}$ of $G$ we also have
\[
			G/\tilde{G}\cong [\mathbf{L},\mathbf{L}_K;\mathbf{R},\mathbf{R}_K]_{\phi} / [\tilde{\mathbf{L}},\tilde{\mathbf{L}}_K;\tilde{\mathbf{R}},\tilde{\mathbf{R}}_K]_{\phi}.
\]
In particular, we have
\[
			(\mathbf{L},\mathbf{L};\mathbf{R},\mathbf{R}) / (\tilde{\mathbf{L}},\tilde{\mathbf{L}};\tilde{\mathbf{R}},\tilde{\mathbf{R}}) \cong \mathbf{L}/\tilde{\mathbf{L}} \times \mathbf{R}/\tilde{\mathbf{R}}.
\]

Elements of $\SOr_4$ of the form $\varphi(l,1)$ and $\varphi(1,r)$ for $l,r \in \SU_2$ are called \emph{left}- and \emph{rightscrews}, respectively. They commute mutually and act freely on $S^3$. 

\begin{lem}\label{lem:rot_crit} Let $l,r \in \SU_2$. Then $\varphi(l,r)$ is a rotation if and only if $\mathrm{Re}(l)=\mathrm{Re}(r)\notin\{\pm 1\}$ where $r$ and $l$ are considered as unit-quaternions.
\end{lem}
\begin{proof}
For $l,r \in \SU_2$ there exist $a,b \in \SU_2$ and $\alpha,\beta \in \R$ such that $a^{-1}la=\cos(\alpha)+\sin(\alpha)i$ and $brb^{-1}=\cos(\beta)+\sin(\beta)i$. Then, with respect to the basis $\mathbb{B}=\{ab,aib,ajb,akb\}$ of $\R^4$, we have
\[
	\begin{array}{ccl}
		\varphi(l,1)_{\mathbb{B}}  = \left(
		  \begin{array}{cc}
		    R(\alpha) & 0 \\
		    0 & R(\alpha) \\
		  \end{array}
		\right),
  \end{array}
  \begin{array}{ccl}
		\varphi(1,r)_{\mathbb{B}}  = \left(
		  \begin{array}{cc}
		    R(\beta) & 0 \\
		    0 & R(-\beta) \\
		  \end{array}
		\right),
  \end{array}
\]
where $R(\alpha)$ is a rotation by the angle $\alpha$. Hence,
\[
	\begin{array}{ccl}
		\varphi(l,r)_{\mathbb{B}}  = \left(
		  \begin{array}{cc}
		    R(\alpha+\beta) & 0 \\
		    0 & R(\alpha-\beta) \\
		  \end{array}
		\right).
  \end{array}
\]
This shows that $\varphi(l,r)$ is a rotation if and only if precisely one of the angles $\alpha+\beta$ and $\alpha-\beta$ is contained in $2\pi \Z$. This is the case if and only if $\cos(\alpha)=\cos(\beta)\neq 1$. Note that the real part of a quaternion is invariant under conjugation by $\SU_2$. Therefore, we have $\mathrm{Re}(l)=\cos(\alpha)$, $\mathrm{Re}(r)=\cos(\alpha)$ and the claim follows.
\end{proof}

In the notation above a Poincar\'e group is given by $P=(\mathbf{C}_{2},\mathbf{C}_{2};\ico ,\ico)< \SOr_4$. Let us now prove Lemma \ref{lem:normal_subgroup_generated_by_ps_in_dim4} and Lemma \ref{lem:char_p_in_SO(4)}.

\begin{proof}[Proof of Lemma \ref{lem:normal_subgroup_generated_by_ps_in_dim4}]
Assume that $G$ is a finite subgroup of $\SOr_4$ such that the quotient group $G / G_{\rota}$ is nontrivial and perfect. In accordance with Proposition \ref{prp:so_subgroups} we represent this group as $G=(\mathbf{L},\mathbf{L}_K;\mathbf{R},\mathbf{R}_K)_{\phi}$ with $-1\in \mathbf{L},\mathbf{R}$. Likewise, we represent $G_{\rota}$ as $G_{\rota}=(\tilde{\mathbf{L}},\tilde{\mathbf{L}}_K;\tilde{\mathbf{R}},\tilde{\mathbf{R}}_K)_{\tilde{\phi}}$ for normal subgroups $\tilde{\mathbf{L}}\triangleleft\mathbf{L}$ and $\tilde{\mathbf{R}}\triangleleft\mathbf{R}$ with $-1\in \tilde{\mathbf{L}},\tilde{\mathbf{R}}$. Since $G$ is perfect, for all $k>0$ we have 
\[
	G / G_{\rota}\cong \Gamma / \tilde{\Gamma} =D^k(\Gamma / \tilde{\Gamma})\cong D^k(\Gamma) /(D^k(\Gamma)\cap \tilde{\Gamma})
\]
where $\Gamma= [\mathbf{L},\mathbf{L}_K;\mathbf{R},\mathbf{R}_K]_{\phi}$ and  $\tilde{\Gamma} = [\tilde{\mathbf{L}},\tilde{\mathbf{L}}_K;\tilde{\mathbf{R}},\tilde{\mathbf{R}}_K]_{\tilde{\phi}}$, and where $D^k$ denotes the $k$th iterated commutator subgroup operator. Moreover, since $\cyc_n$, $\dih_n$, $\sym_4$, $\alt_4$ and $\mathbf{C}_2$ are solvable, so are $\mathbf{C}_n$, $\mathbf{D}_n$, $\mathbf{T}$ and $\mathbf{O}$. Hence, we have, perhaps after interchanging the factors, $\mathbf{R}=\mathbf{I}$. For, otherwise $D^k(\Gamma)$ would be trivial for sufficiently large $k$.

Since the alternating group $\mathfrak{A}_5$ is simple, the only normal subgroups of $\ico$ are the trivial subgroup $\mathbf{C}_{1}$, $\ico$ itself and its center $\mathbf{C}_{2}$. We claim that $\mathbf{R}_K=\ico$. Otherwise, we would either have $G=(\mathbf{I},\mathbf{C}_2;\mathbf{I},\mathbf{C}_2)$ or $G=(\mathbf{I},\mathbf{C}_1;\mathbf{I},\mathbf{C}_1)$. The group $(\mathbf{I},\mathbf{C}_1;\mathbf{I},\mathbf{C}_1)$ is generated by rotations by Lemma \ref{lem:rot_crit}. Hence, in both cases $G/G_{\rota}$ has order $\leq 2$ and can thus not be nontrivial and perfect. Therefore, we have $\mathbf{R}_K=\ico$ and $G$ is one of the groups $(\mathbf{C}_{2n},\mathbf{C}_{2n};\ico ,\ico)$, $(\mathbf{D}_{n},\mathbf{D}_{n};\ico ,\ico)$, $(\mathbf{T},\mathbf{T};\ico ,\ico)$, $(\mathbf{O},\mathbf{O};\ico ,\ico)$ or $(\ico ,\ico;\ico ,\ico)$.

Let us examine the possibilities for $G_{\rota}$ in these cases. In the case $\tilde{\mathbf{R}}=\tilde{\mathbf{R}}_K=\ico$ we would also have $\mathbf{C}_2<\tilde{\mathbf{L}}=\tilde{\mathbf{L}}_K<\mathbf{L}<\mathbf{I}$ and so the quotient group $G/G_{\rota}\cong \mathbf{L}/\tilde{\mathbf{L}}$ would be solvable. Hence, this case cannot occur. In the cases $\tilde{\mathbf{R}}=\ico$ and $\tilde{\mathbf{R}}_K=\mathbf{C}_{i}$, $i=1,2$, we would have $G_{\rota}=(\ico,\mathbf{C}_{i};\ico,\mathbf{C}_{i})$ and hence $G=(\ico,\ico;\ico,\ico)$. However, this would contradict the fact that $G_{\rota}$ is normal in $G$ since $\ico$ has conjugacy classes (in $\ico$) with more than two elements. Consequently, we must have $\tilde{\mathbf{R}}=\mathbf{C}_{2}$. Since transformations of the form $\varphi((l,\pm1))$ are never rotations by Lemma \ref{lem:rot_crit}, we deduce that $G_{\rota}=(\mathbf{C}_{2},\mathbf{C}_{1};\mathbf{C}_{2},\mathbf{C}_{1})=\{\mathrm{id}\}$ and thus that $G$ is perfect itself. The only nontrivial perfect groups among the remaining possibilities for $G$ are $(\mathbf{C}_{2},\mathbf{C}_{2};\ico ,\ico)$ and $(\ico ,\ico;\ico ,\ico)$. The latter case can be excluded since $(\ico ,\ico;\ico ,\ico)$ contains rotations by Lemma \ref{lem:rot_crit}, but $G_{\rota}=\{\mathrm{id}\}$ as shown above. Hence, we conclude that $G$ is given by $(\mathbf{C}_{2},\mathbf{C}_{2};\ico ,\ico)$, a Poincar\'e group isomorphic to $\SL_2(5)$.
\end{proof}

\begin{proof}[Proof of Lemma \ref{lem:char_p_in_SO(4)}] Similarly as in the preceding proof, the fact that $G\cong \SL_2(5)$ is a prefect group implies that the only possibilities for $G$ are $(\mathbf{C}_{2},\mathbf{C}_{2};\ico ,\ico)$ and $(\ico ,\ico;\ico ,\ico)$. But only $(\mathbf{C}_{2},\mathbf{C}_{2};\ico ,\ico)$ is isomorphic to $\SL_2(5)$ and so the claim follows.
\end{proof}

\section{When is the underlying space of an orbifold a topological manifold?}
\label{sec:if_dir}

According to Lemma \ref{lem:lipschitz_orbifold} the question in this section's title amounts to proving Theorem A. We begin by verifying the if direction of Theorem A and Corollary B. Let $G<\Or_n$ be a product of a reflection-rotation group $G_{\mathrm{rr}}<\Or(V_{\mathrm{rr}})$ and a finite number of Poincar\'e groups $P_i<\SOr(V_i)$, $i=1,\ldots,k$, that act in pairwise orthogonal spaces $V_{\mathrm{rr}}$ and $V_i$, $i=1,\ldots,k$, as in the statement of Theorem A or Corollary B. We can assume that $V_{\mathrm{rot}}=\R^{n-4k}$ and $V_i=\R^4$. Elementary topological arguments show that
\[
\begin{array}{rcl}
  S^{n-1}/G&\cong &S^{n-1-4k}/G_{\mathrm{rr}}* S^3/P_1 * \cdots * S^3/P_k, \\		
		\R^{n}/G&\cong &\R^{n-4k}/G_{\mathrm{rr}}\times \R^4/P_1 \times \cdots \times \R^4/P_k.
\end{array}
\]
By Theorem \ref{thm:theorem_Lange} we have $S^{n-1-4k}/G_{\mathrm{rr}} \cong B^{n-1-4k}$ or $S^{n-1-4k}/G_{\mathrm{rr}} \cong S^{n-1-4k}$ depending on whether $G_{\mathrm{rr}}$ contains a reflection or not. Also note that $S^n*S^m\cong S^{n+m+1}$ and $S^n*B^m\cong B^{n+m+1}$. Moreover, recall from Lemma \ref{lem:poinc_group_double} that, for a Poincar\'e group $P < \SOr_4$, the space $\R \times \R^4/P$ is homeomorphic to $\R^5$, and that the join $S^l * S^3/P$ is homeomorphic to $S^{l+4}$ by the double suspension theorem if $l\geq 1$. Given our additional assumptions on $n$ depending on $k$ in Theorem A and Corollary B, this implies our claims in the case $k\leq 1$. 

For a finite subgroup $G<\Or_n$ we have $C(S^{n-1}/G)\cong \R^n/G$. Moreover, the join operation is associative and satisfies $C(X*Y)\cong CX *CY$ \cite[Prop.~I.5.15]{MR1744486}. Therefore, in the case $k>1$ our claims reduce to an application of the double suspension theorem and the following corollary of it.

\begin{lem}
Let $P_1,P_2 < \SOr_4$ be Poincar\'e groups. Then $S^3/P_1 * S^3/P_2$ is a $7$-sphere.
\label{lem:two_poinc_groups}
\end{lem}
\begin{proof} We can assume that $S^3/P_1$ and $S^3/P_2$ are triangulated by simplicial complexes $K_1$ and $K_2$. Then $K=K_1*K_2$ is naturally a simplicial complex homeomorphic to $X:=S^3/P_1 * S^3/P_2$ (see e.g. \cite{MR0350744}). Since $S^3/P_1$ and $S^3/P_2$ are manifolds, all links of vertices of $K_1$ and $K_2$ are topological $2$-spheres. Hence, for a vertex $x\in K$ we have
\[
		|\mathrm{link}_K(x)|\cong S^2*S^3/P \cong \Sigma^3(S^3/P) \cong S^6
\]
by the double suspension theorem, $P<\SOr_4$ being a Poincar\'e group. Therefore $X$ is a topological manifold. Because of $CX\cong C(S^3/P_1) \times C(S^3/P_2)$ the space $X$ has the integral homology groups of a sphere by Lemma \ref{lem:characterization_of_homology_manifolds}. Moreover, $X$ is simply connected as the join of two path-connected spaces. As a compact, simply connected topological manifold with the integral homology of a sphere the space $X$ has to be a topological $7$-sphere by the generalized Poincar\'e conjecture \cite{Newman}.
\end{proof}

Hence, we have verified the if directions of Theorem A and Corollary B. To prove the only if directions we begin by introducing some organizing concepts.

\subsection{Minimal subgroups} \label{sec:homology_criterion}
Let $G<\Or_n$ be a finite subgroup. For a subgroup $H<G$ and a linear subspace $L\subset \R^n$ we introduce the following notations
\[
\begin{array}{rcl}
  I(L):=G_L   &:=  & \{g\in G\mid\forall x\in L : gx=x\} \\
		L(H):=\text{Fix}(H)   &:=  & \{x\in\R^n\mid g\in G: gx=x\}. \\
\end{array}
\]
Moreover, we write
\[
\begin{array}{rcl}
	  \mathfrak{I}:=\mathfrak{I}(G)   &:=  & \{I(L)\mid\text{ L is a linear subspace of }\R^n\} \\
		\mathfrak{L}:=\mathfrak{L}(G)   &:=  & \{L(H) \mid H \text{ is a subgroup of } G\}. \\
\end{array}
\]
Inclusion induces partial orders on $\mathfrak{L}$ and $\mathfrak{I}$. The group $G$ acts by translation on $\mathfrak{L}$ and by conjugation on $\mathfrak{I}$. The correspondence
\[
\begin{array}{lccc}
  &\mathfrak{I}   & \stackrel{1:1}{\longleftrightarrow}  & \mathfrak{L} \\
   L : 						& H 					& \LMTo  				&  \mathrm{Fix}(H) \\
   I: 						& G_L		& \longmapsfrom &  L,
\end{array}
\]
is order-reversing, one-to-one and $G$-equivariant. Since $G$ is finite, so are $\mathfrak{I}$ and $\mathfrak{L}$. Note that $I(L)$ is by definition the maximal subgroup of $G$ that fixes $L$ pointwise. The equivalence classes under the action of $G$ are in one-to-one correspondence to the strata of $\R^n/G$ and the isotropy group of a point in a stratum is given by the corresponding subgroup in $\mathfrak{I}$. We say that a subgroup $H \in \mathfrak{I}$ is \emph{minimal} if it is a nontrivial minimal subgroup in $\mathfrak{I}$ with respect to inclusion. Corresponding subspaces are called \emph{maximal} subspaces in $\mathfrak{L}$. Conjugacy classes of minimal subgroups and maximal subspaces correspond to strata of $\R^n/G$ that are not contained in the closure of any higher dimensional singular stratum. We write
\[
 \mathfrak{L}_{\maxi}  :=   \{L\in \mathfrak{L} \mid L\text{ maximal}\},\; \mathfrak{I}_{\mini}  :=   \{H\in \mathfrak{I} \mid H\text{ minimal}\}. 
\]
Moreover, we denote the subgroup of $G$ generated by all minimal subgroups in $\mathfrak{I}$ by $G_{\mini}$. We record some properties of minimal subgroups and $G_{\mini}$ in the following lemmas.

\begin{lem}
The group $G_{\mini}$ is normal in $G$. If $G$ is nontrivial, then so is $G_{\mini}$.
\label{lem:min_nonempty}
\end{lem}
\begin{proof}The subgroup $G_{\mini}$ of $G$ is normal in $G$ since $\mathfrak{I}_{\mini}\subset \mathfrak{I}$ is closed under conjugation by $G$. If there are no proper subgroups of $G$ that are minimal in $\mathfrak{I}$, then $G$ itself is minimal. Hence, if $G$ is nontrivial, then so is $G_{\mini}$.
\end{proof}

\begin{lem} For any point $x \in \R^n-\{0\}$ we have $(G_x)_{\mini} \subseteq (G_{\mini})_x$. 
\label{lem:stabilizer_minimizer_inclusion}
\end{lem}
\begin{proof} Every minimal subgroup in $\mathfrak{I}(G_x)$ is also minimal in $\mathfrak{I}(G)$. Hence, $(G_x)_{\mini} \subseteq G_{\mini} \cap G_x = (G_{\mini})_x$.											
\end{proof}

\begin{lem}
Let $H \in \mathfrak{I}$ be nontrivial. Then the action $H \curvearrowright S^{d-1} \subset \R^d = L(H)^{\bot}$ is free if and only if $H$ is minimal.
\label{lem:free_minimal}
\end{lem}
\begin{proof} The action $H \curvearrowright S^{d-1}$ being free means that there is no nontrivial $h\in H$ that fixes a subspace strictly larger than $L(H)$. This is the case if and only if $H$ is minimal.
\end{proof}

Now we deduce necessary conditions on minimal subgroups in $\mathfrak{I}$ in order for $\R^n/G$ to be a homology manifold. 

\begin{lem} Let $G<\Or_n$ be a nontrivial finite subgroup. Assume that $\R^n/G$ is a homology manifold. Then $G < \SOr_n$ and for every minimal subgroup $H \in \mathfrak{I}$ the quotient $S^{d-1}/H$ of the action $H \curvearrowright S^{d-1} \subset \R^d = L(H)^{\bot}$ is a homology sphere.
\label{lem:min_implies_hom_sphere}
\end{lem}
\begin{proof}
By Lemma \ref{lem:characterization_of_homology_manifolds} the quotient $S^{n-1}/G$ has the homology groups of a sphere and is again a homology manifold. The first condition is global and implies $G < \SOr_n$ by Lemma \ref{lem:ac_hom_quo}. The second condition is local and says that each point $p= \pi(x)\in S^{n-1}/G$ has a neighborhood which is a homology manifold. A small neighborhood of $x$ is homeomorphic to $T_x S^{n-1} / G_x \cong \R^{n-1}/G_x$ via the exponential map. Hence, also $T_x S^{n-1} / G_x$ is a homology manifold for every $x \in S^{n-1}$. Proceeding iteratively we find that for each $H\in \mathfrak{I}$ the quotient space of the action $H\curvearrowright \R^d = L(H)^{\bot}$ is a homology manifold. Moreover, if $H$ is a minimal subgroup of $G$, then $S^{d-1}/H$ is also a manifold since, in this case, the action $H \curvearrowright S^{d-1}$ is free by Lemma \ref{lem:free_minimal}. Since $S^{d-1}/H$ has the integral homology groups of a sphere by Lemma \ref{lem:characterization_of_homology_manifolds}, $(ii)$, it is a homology sphere as claimed.
\end{proof}

Now we can show that only very special minimal subgroups occur.

\begin{prp} Let $G<\Or_n$ be a nontrivial finite subgroup. Assume that $\R^n/G$ is a homology manifold. Then, for every $H \in \mathfrak{I}_{\mini}$, denoting $d=\mathrm{codim}L(H)$, either:
\begin{compactenum}
	\item $d=2$ and $H=C_k < \SOr(L(H)^{\bot})$, a cyclic group of order $k$ for some $k\geq 2$.
	\item $d=4$ and $H=P < \SOr(L(H)^{\bot})$, a Poincar\'e group.
\end{compactenum}
\label{prp:characterization_minimal_subgroup_Top}
\end{prp}
\begin{proof}
By Lemma \ref{lem:min_implies_hom_sphere} we have $G< \SOr_n$ and the quotient $S^{d-1}/H$ of the free action $H \curvearrowright S^{d-1} \subset \R^d = L(H)^{\bot}$ is a homology sphere with fundamental group isomorphic to $H$. In the case $d=2$ the group $H$ has to be a cyclic group $C_k < \SOr(L(H)^{\bot})$ of order $k$ for some $k\geq 2$. The case $d=3$ cannot occur since every nontrivial subgroup of $\SOr_3$ has fixed points on $S^2$. In the case $d \geq 4$ it follows either from Zassenhaus's result \cite{MR3069653,MR3069657} (cf. \cite[Thm.~6.2.1]{MR928600}) and the fact that $H_3(\SL_2(p);\Z) \neq 0$ (cf. proof of Proposition \ref{prp:char_sl_homology}) or by Theorem \ref{thm:Sjerve} that $d=4$ and that $H$ is isomorphic to $\SL_2(5)$. Therefore $H$ is a Poincar\'e group by Lemma \ref{lem:char_p_in_SO(4)} and so the proposition is proven.
\end{proof}

Let $G_{\rota}$ be the subgroup of $G$ generated by all rotations contained in $G$. The subgroup $G_{\rota}$ is normal in $G$ since the set of all rotations in $G$ is invariant under conjugation by $G$. Set $\mathfrak{I}_P:=\{H\in \mathfrak{I}_{\mini}| \mathrm{codim}(L(H))=4 \}$. We summarize the results of this section in a proposition.

\begin{prp} Let $G<\Or_n$ be a nontrivial finite subgroup. Assume that $\R^n/G$ is a homology manifold. Then $G<\SOr_n$, and the group $G_{\mini}$ is nontrivial and normal in $G$. Moreover, the set $\mathfrak{I}_P$ is invariant under conjugation by $G$, it consists of a finite number of Poincar\'e groups and we have $G_{\mini}=\left\langle G_{\rota},\mathfrak{I}_P \right\rangle$. In particular, for a rotation group $G$ we have $G=G_{\mini}$.
\label{prp:minsub_generators}
\end{prp}
\begin{proof}The first two claims simply repeat the statements of Lemma \ref{lem:min_nonempty} and Lemma \ref{lem:min_implies_hom_sphere}. By Proposition \ref{prp:characterization_minimal_subgroup_Top} every minimal subgroup $H$ with $\mathrm{codim}L(H)=2$ is contained in $G_{\rota}$. Conversely, because of $G<\SOr_n$ and the fact that any transformation $g\in \Or_n$ with $\mathrm{codim}\text{Fix}(G)=1$ reverses the orientation, every rotation in $G$ belongs to a minimal subgroup $H$ with $\mathrm{codim}L(H)=2$. Hence, $G_{\rota}$ is generated by all minimal subgroups $H$ with $\mathrm{codim}L(H)=2$. Since the set of all minimal subgroups is finite and invariant under conjugation by $G$, so is $\mathfrak{I}_P$. By Proposition \ref{prp:characterization_minimal_subgroup_Top} each $P\in\mathfrak{I}_P$ is a Poincar\'e group. The claim follows since there are no other minimal subgroups by Proposition \ref{prp:characterization_minimal_subgroup_Top}.
\end{proof} 

\subsection{Orthogonal splitting}\label{sub:orthogonal_splitting}
In this section we show that the subgroups $G_{\rota}$ and $P \in \mathfrak{I}_P$ occurring in Proposition \ref{prp:minsub_generators} act in pairwise orthogonal subspaces. More precisely, we show that $\R^n$ splits as an orthogonal sum of subspaces 
\[
	\R^n = V_{\rota} \oplus \bigoplus_{L \in \mathfrak{L}_P} L^{\bot}
\]
where we have set $\mathfrak{L}_P:=\{L\in \mathfrak{L}\mid \mathrm{codim}(L)=4 \}=\{L(P)\mid P\in \mathfrak{I}_P \}$, and $V_{\rota}$ to be the span of $V_0:=\text{Fix}(G_{\mini})$ and the orthogonal complements of all maximal subspaces in $\mathfrak{L}$ of codimension two.
\begin{figure}
	\centering
		\def\svgwidth{0.6\textwidth}
		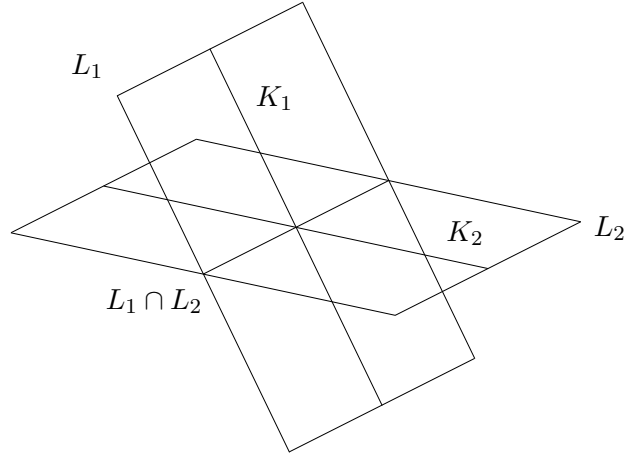
	\caption{Intersection of two fixed-point subspaces.}
	\label{fig:asymptotic-geodesics}
\end{figure}

We begin by showing that the Poincar\'e groups act in pairwise orthogonal spaces. The claim that the rotation group $G_{\rota}$ acts in a space orthogonal to all of them can then be reduced to this case, see Lemma \ref{lem:4_and_2}. The idea of the proof is to reach a contradiction to the finiteness of $\mathfrak{L}_P$ on the assumption that there are two Poincar\'e groups that do not act in orthogonal spaces. We achieve this by defining an auxiliary function $D:\mathfrak{L}_P\times\mathfrak{L}_P \To [0,5]$ and by showing that, if there are distinct subspaces $L_1,L_2 \in \mathfrak{L}_P$ whose orthogonal complements are not orthogonal, then there exists another such pair $L'_1,L'_2 \in \mathfrak{L}_P$ with $D(L'_1,L'_2)<D(L_1,L_2)$ (see Lemma \ref{lem:poincares_are_orthogonal}).

More precisely, we proceed as follows. For $L_1,L_2 \in \mathfrak{L}_P$ we set $K_{i} := (L_1\cap L_2)^{\bot_{L_i}}$, $i=1,2$. Here $\bot_{L_i}$ denotes the orthogonal complement in $L_i$. We can decompose $L_1$ and $L_2$ into orthogonal sums
\[
		L_1 = \left(L_1\cap L_2\right) \oplus K_{1} \text{ and } \ L_2 = \left(L_1\cap L_2\right) \oplus K_{2}.
\]
The function $D$ is defined as
\[
		D(L_1,L_2) := \text{dim}K_{1}+d(K_{1},K_{2})
\]
where
\[
		d(K_1,K_2) := \begin{cases}
  									0   & \text{if dim }K_1 \cdot \text{dim }K_2=0\\
  									\dfrac{2}{\pi} \angle(K_1,K_2) 
  									 & \text{otherwise,}
							\end{cases}
\]
and
\[
		\angle(K_1,K_2)  :=  \mathrm{min} \left\{ \angle(v_1,v_2) \mid 0\neq v_1 \in K_1, 0\neq v_2 \in K_2 \right\}.  				
\]
Here $\angle(v_1,v_2)$ denotes the angle between the vectors $v_1$ and $v_2$ measured in radians. We record some properties in a lemma.
\begin{lem} \label{lem:characterization_of_D}For subspaces $L_1,L_2 \in \mathfrak{L}_P$ the following properties are satisfied.
\begin{compactenum}
\item $K_1\cap K_2=\{0\}$ and $\mathrm{dim}K_{1}=\mathrm{dim}K_{2}\in\{0,1,2,3,4\}$.
\item $d(K_{1},K_{2})\in [0,1]$, and $d(K_{1},K_{2})=1 \Leftrightarrow (K_1 \bot K_2\text{ and } K_i \neq \{0\}, i=1,2)$.
	\item $0\leq D(L_1,L_2)\leq 5$.
	\item $D(L_1,L_2)=0 \Leftrightarrow L_1=L_2\Leftrightarrow d(K_1,K_2)=0$.
	\item $D(L_1,L_2)=5 \Leftrightarrow \left(\mathrm{dim}K_{1}=\mathrm{dim}K_{2}=4 \text{ }\mathrm{and}\text{ } K_{1} \bot K_{2}\right) \Leftrightarrow L_1^{\bot}\bot L_2^{\bot}$.
\end{compactenum}
\end{lem}
\begin{proof}
$(i)$ The inclusions $K_i\subseteq L_i$, $i=1,2$, imply $K_1\cap K_2 \subseteq L_1\cap L_2$. On the other hand, $K_1\cap K_2$ is orthogonal to $L_1\cap L_2$. Hence, the first claim follows. Because of $\mathrm{codim}L_{1}=\mathrm{codim}L_{2}=4$ we have $\mathrm{codim}(L_{1}\cap L_{2})\in \{4,5,6,7,8\}$. Therefore, the second claim follows from 
 $\mathrm{dim}K_{i}=\mathrm{dim}L_{i}-\mathrm{dim}(L_{1}\cap L_{2})=\mathrm{codim}(L_{1}\cap L_{2})-4$, $i=1,2$.

$(ii)$ This is clear from the definition of $d$.

$(iii)$ This follows from $(i)$ and $(ii)$.

$(iv)$ We have $D(L_1,L_2)=0$ if and only if $\text{dim }K_{1}=0$. This is the case if and only if $L_1=L_2$. Because of $K_1\cap K_2=\{0\}$ by $(i)$, the condition $\text{dim }K_{1}=0$ is also equivalent to $ d(K_1,K_2)=0$

$(v)$ We have $D(L_1,L_2)=5$ if and only if $\text{dim }K_{1}=4$ and $d(K_{1},K_{2})=1$. Hence, the first equivalence follows from $(ii)$. In this case we have $K_{1} \bot K_{2}$. This implies $K_{1}=L_2^{\bot}$ and $K_{2}=L_1^{\bot}$ by dimension reasons and thus also $L_1^{\bot}\bot L_2^{\bot}$. On the other hand, $L_1^{\bot}\bot L_2^{\bot}$ implies $K_{1}=L_2^{\bot}$ and $K_{2}=L_1^{\bot}$ and so the converse holds, too.
\end{proof}
To pursue our strategy we need the following elementary geometric lemma.
\begin{lem}\label{lem:spherical_comparison}
Let $V$ be a Euclidean vector space with a proper subspace $W$ and let $U = W^{\bot}$ be the orthogonal complement of $W$ in $V$. Let $v= w+u$ be the orthogonal decomposition with respect to $W$ and $U$ of a unit-vector $v \in V$ with $0 \neq u \in U$. Let $\phi \in \Or(V)$ be such that $\phi_{|W}=\mathrm{id}_W$. If $\alpha:=\angle( u, \phi(u)) \leq \pi/3$, then
$
		\gamma:=\angle(v,\phi(v)) < \angle(v,W)=: \beta.
$
\label{lem:angle_reduction}
\end{lem}
\begin{figure}
	\centering
		\def\svgwidth{0.7\textwidth}
		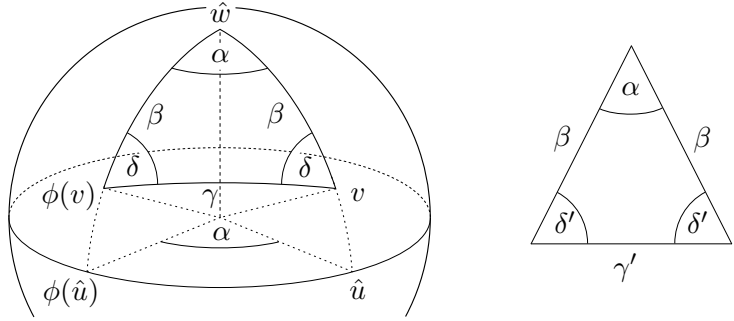
	\caption{Left: An isosceles spherical triangle with opening angle $\alpha\leq \pi/3$, leg length $\beta<\pi/4$ and base length $\gamma$. Right: An isosceles Euclidean comparison triangle with opening angle $\alpha$, leg length $\beta$ and base length $\gamma'$. In the figure we have set $\hat{u}=\frac{u}{\left\|u\right\|}$ and $\hat{w}=\frac{w}{\left\|w\right\|}$.}
	\label{fig:spherical-triangle}
\end{figure}
\begin{rem}\label{rem:strict_ineq} It is crucial for our proof to have the strict inequality $\gamma < \beta$, see proof of Lemma \ref{lem:poincares_are_orthogonal}. In view of Remark \ref{rem:not_in_subspace} also note that the statement of the lemma is wrong for every $\alpha > \pi/3$. Namely, in this case it fails for sufficiently small $\beta$.
\end{rem}
\begin{proof} The lemma can be proven by abstract calculations with inner products. Here, we give a proof based on comparison geometry. It better reflects the geometric meaning of the lemma, which was brought to our attention by Petrunin. Roughly speaking, the reason why the lemma is true is the fact that the unit sphere in $\R^3$ with its induced Riemannian metric has strictly positive curvature.

For $w=0$ or $\phi(v)=v$ the claim is trivially true. Otherwise, the vectors $\hat{w}=\frac{w}{\left\|w\right\|}$, $v$ and $\phi(v)$ span a three-dimensional subspace of $V$ and lie on an open hemisphere of the unit sphere $S^2$ of this subspace as depicted in Figure \ref{fig:spherical-triangle}. In particular, they can be connected by unique minimizing geodesics on $S^2$ and form the vertices of a nondegenerate spherical triangle in $S^2$. The angle $\alpha$ appears as the angle at $\hat w$ of this spherical triangle. The angles $\beta$ and $\gamma$ appear as side lengths $\beta=d_{S^2}(\hat{w},v)=d_{S^2}(\hat{w},\phi(v))$ and $\gamma = d_{S^2}(v,\phi(v))$ of this triangle, cf. Figure \ref{fig:spherical-triangle}, where $d_{S^2}$ denotes the induced length metric on $S^2$. We can draw an isosceles Euclidean comparison triangle in $\R^2$ with opening angle $\alpha$ and leg length $\beta$. Our assumption $\alpha\leq \pi/3$ implies that the base length $\gamma'$ of this Euclidean triangle satisfies $\gamma'\leq \beta$. Hence, in order to prove the lemma it remains to show that $\gamma<\gamma'$. This follows from a strict version of Rauch's comparison theorem and the fact that the curvature of $S^2$ is strictly positive \cite[Ch.~10.2]{MR1138207}. Here, we give a more elementary argument using trigonometry. By the Euclidean and the spherical law of sine \cite{MR1802706}, applied to our triangles, we have
\[
	 \frac{\sin(\alpha)}{\gamma'}=\frac{\sin(\delta')}{\beta} \text{ and } \frac{\sin(\alpha)}{\gamma}=\frac{\sin(\delta)}{\beta},
\]
and hence $\sin(\delta)\gamma = \sin(\delta')\gamma'$. In particular, we have $\gamma<\gamma'$ if and only if $\sin(\delta) > \sin(\delta')$. Now, the surface area of our spherical triangle is $A=\alpha+2\delta-\pi$ \cite{MR1802706}. It is bounded from above by the surface area $B$ of the spherical triangle defined by $\hat w$, $\hat u$ and $\phi(\hat u)$, which is $B=\frac{1}{2}\frac{\alpha}{2\pi} \cdot 4\pi=\alpha$. This implies $\alpha+2\delta -\pi =A\leq B= \alpha$ and thus $\delta\leq \pi/2$. Therefore, the condition $\sin(\delta) > \sin(\delta')$ is equivalent to the condition $\delta > \delta'$ by the strict monotonicity of the sine function on $[0,\pi/2]$. Since the sum of all angles of a nondegenerate spherical triangle is strictly bigger than $\pi$, we have $\alpha+2\delta>\pi=\alpha+2\delta'$. This implies $\delta > \delta'$ and so the claim follows.
\end{proof}
In order to apply this lemma we need to understand the orbit geometry of the action $P \curvearrowright S^3$ of a Poincar\'e group $P<\SOr_4$. First note that all orbits are isometric since the canonical metric on $S^3$ is $S^3$-bi-invariant. So let $x\in S^3$ be an arbitrary point and let $X=Px$ be its orbit under the action of $P$. Since $P$ is finite, the sets $X_{\alpha}=\{y\in X|\angle(x,y)=\alpha \}\subset X$ are empty except for a finite number of values of $\alpha$. The points of $X_{\alpha}$ lie in the intersection of $S^3$ with a hyperplane of $\R^4$. This intersection is a point for $\alpha\in\{0,\pi\}$ and a two-dimensional sphere for $\alpha=(0,\pi)$. From the explicit coordinate representation of $P$ described in Section \ref{subsec:the_binary_icosahedral_group} one can read off the structure of $X=\bigcup_{\alpha\in[0,\pi]}X_{\alpha}$, which is summarized in Table \ref{tbl:structure_600cell} \cite[p.~98]{Buekenhout}. The last column specifies the geometric structure of the points in $X_{\alpha}$. The points in $X_{\pi/5}$ form the vertices of an icosahedron. Therefore, their affine span is three-dimensional. Since every point in $X_{\pi/3}$ is not contained in this affine subspace, the following lemma holds true.
\begin{lem}
For any $\alpha\in [\pi/3,\pi]$ the set $\bigcup_{\beta\in(0,\alpha]}X_{\beta}$ is not contained in an affine three-dimensional subspace of $\R^4$.
\label{lem:not_in_subspace}
\end{lem}
\begin{rem}\label{rem:not_in_subspace} Note that the lemma is wrong for any $\alpha<\pi/3$. For our proof it is crucial that the set of $\alpha$ to which both Lemma \ref{lem:angle_reduction} and Lemma \ref{lem:not_in_subspace} apply is nonempty; see proof of Lemma \ref{lem:poincares_are_orthogonal}. In fact, it consists of the single value $\alpha=\pi/3$.
\end{rem}

\renewcommand{\arraystretch}{1.2}
\begin{table}[t]
\begin{tabular}{ c | c | c }                        
  $\alpha$ & $|X_{\alpha}|$ & $X_{\alpha}$ \\
  \hline 
  $0,\pi$  & $1$  & point \\
  $\pi/5,4\pi/5$ & $12$ & icosahedron \\
  $\pi/3,2\pi/3$ & $20$ & dodecahedron \\
  $2\pi/5,3\pi/5$ & $12$ & icosahedron \\
  $\pi/2$ & $30$ & icosidodecahedron \\
\end{tabular}
\caption{Geometric structure of the $600$-cell.}
\label{tbl:structure_600cell}
\end{table}\renewcommand{\arraystretch}{1}
Now we can prove the first splitting lemma.
\begin{lem}
For distinct $L_1,L_2 \in \mathfrak{L}_P$ we have $D(L_1,L_2)=5$. That is, $L_1^{\bot}\bot L_2^{\bot}$ holds by Lemma \ref{lem:characterization_of_D}. Thus the corresponding Poincar\'e groups $P_1,P_2<G$ act in orthogonal spaces.
\label{lem:poincares_are_orthogonal}
\end{lem}
\begin{proof}
Suppose we have distinct $L_1,L_2 \in \mathfrak{L}_P$ with $D(L_1,L_2)<5$. Since $L_1$ and $L_2$ are distinct, we also have $0<D(L_1,L_2)$ by Lemma \ref{lem:characterization_of_D}, $(iv)$. We are going to find another pair $L_1',L_2' \in \mathfrak{L}_P$ with
\[
		0<D(L_1',L_2')<D(L_1,L_2)<5.
\]
By iteration this yields a contradiction to the finiteness of $\mathfrak{L}_P$ and thereby will prove the lemma.

Since $L_1$ and $L_2$ are distinct, the subspaces $K_{1}= (L_1 \cap L_2)^{\bot_{L_1}}$ and $K_{2}=(L_1 \cap L_2)^{\bot_{L_2}}$ are nontrivial. Therefore, we can choose $v_1 \in K_{1}$ and $v_2 \in K_{2}$ with $\left\|v_1\right\|=\left\|v_2\right\|=1$ such that the angle $\angle(v_1, v_2)$ is minimal and such that $d(K_1,K_2)=\frac{2}{\pi}\angle(v_1, v_2)$ holds. We decompose $v_2$ into orthogonal components with respect to $L_1$ and $L_1^{\bot}$, that is, $v_2=w+u$ where $w \in L_1$ and $u\in L_1^{\bot}$. We claim that $u \neq 0$. Indeed, otherwise we would have $v_2 \in L_1 \cap K_{2} =\{0\}$, a contradiction. 

Since $u$ is nontrivial, the set $S:=\{g\in P_1|0<\angle(gu ,u)\leq \pi/3\} \subset P_1$ is well-defined. By linearity we have
\[
		Sv_2=S(w+u)=w+Su.
\]
By Lemma \ref{lem:not_in_subspace} the smallest affine subspace of $V=\R^n$ that contains $Su$ is four-dimensional and thus given by $L_1^{\bot}$. Therefore, also the smallest affine subspace of $V$ that contains $Sv_2$ is four-dimensional and thus given by  $w+L_1^{\bot}$. Moreover, if $\dim K_{1}=\dim K_{2}=4$, then $w+L_1^{\bot}$ is not a linear subspace, that is $w\neq 0$. Indeed, if $\dim K_{1}=\dim K_{2}=4$, then our assumption $D(L_1,L_2)=\dim K_{1}+d(K_1,K_2)<5$ implies $\frac{2}{\pi}\angle(v_1, v_2)=d(K_1,K_2)<1$ by Lemma \ref{lem:characterization_of_D}. This means that $v_1$ and $v_2$ are not orthogonal and so $w$ is nontrivial in this case.

Since $K_2$ is a linear subspace of $V$ of dimension $\dim K_2 \leq 4$ by Lemma \ref{lem:characterization_of_D}, $(i)$, the preceding paragraph shows that $Sv_2$ is not contained in $K_2$. In particular, there is some $g\in S$ such that $gv_2 \notin K_{2}$. We set $L_1':=gL_2$, $L_2':=L_2$ and claim that $0<D(L'_1,L_2')<D(L_1,L_2)$ holds. This would finish the proof of the lemma as explained above.

Let us first show that $0<D(L_1',L_2')$. By Lemma \ref{lem:characterization_of_D}, $(iv)$, this amounts to showing that $L_2$ and $gL_2$ are distinct. The fact that $g$ fixes $L_1\cap L_2 \subset L_1=\mathrm{Fix}(P_1)$ pointwise and that $v_2\in K_2 \bot (L_1\cap L_2)$ implies $gv_2\bot (L_1 \cap L_2)$. Hence, because of $gv_2 \notin K_{2}$, the vector $gv_2$ is also not contained in $L_2=(L_1 \cap L_2) + K_2$. Now, $v_2\in L_2$ and $gv_2 \notin L_2$ shows that $L_2\neq gL_2$ as desired.

It remains to show that $D(L_2,gL_2)<D(L_1,L_2)$ holds. Set $K_i':=(L_1' \cap L_2')^{\bot_{L_i'}}$, $i=1,2$. Since $g$ fixes $L_1 \cap L_2$ pointwise, we have $L_1 \cap L_2 \subseteq L_2 \cap gL_2$. We distinguish two cases according to whether we have a strict inclusion or not.

Suppose that $L_1 \cap L_2$ is a proper subspace of $L_2 \cap gL_2$. Then we have $\dim K'_1+1 \leq \dim K_1$ and thus
\[
		D(L_2,gL_2)=\dim (K'_1)+d(K'_1,K'_2)\leq D(L_1,L_2)+d(K'_1,K'_2)-d(K_1,K_2)-1.
\]
Our assumption $L_1\neq L_2$ entails $0<d(K_1,K_2)$ by Lemma \ref{lem:characterization_of_D}, $(iv)$. Moreover, also by Lemma \ref{lem:characterization_of_D} we have $d(K'_1,K'_2)\leq 1$. This implies $D(L_2,gL_2)< D(L_1,L_2)$ as desired.

Now suppose that $L_1 \cap L_2 = gL_2 \cap L_2$. In this case we have $K'_{1}=gK_{2}$, $K'_{2}=K_{2}$ and so our claim is equivalent to $d(gK_{2},K_{2})<d(K_{1},K_{2})$. We set $W=L_1$, $U=L_1^{\bot}$, $v=v_2=w+u$ and $\phi=g$. Then all assumptions of Lemma \ref{lem:angle_reduction} are satisfied. We obtain

\[
		d(gK_{2},K_{2}) \leq \frac{2}{\pi} \angle(v_2, gv_2)<\frac{2}{\pi} \angle(v_2, L_1)  = \frac{2}{\pi}\angle(v_2, K_1) = d(K_{1},K_{2})
\]
where we have applied Lemma \ref{lem:angle_reduction} in the strict inequality and the fact that $v_2 \bot (L_1 \cap L_2)$ in the equality $\angle(v_2, L_1)  = \angle(v_2, K_1)$.

Consequently, in any case we have
\[
		0<D(L_1',L_2')<D(L_1,L_2)<5,
\]
and so the lemma follows by contradiction as explained above.
\end{proof}

We have shown that all Poincar\'e groups act in orthogonal spaces. Now we claim that $G_{\rota}$ acts in a space orthogonal to all of them. We can assume that there are maximal subspaces $L_2,L_4 \in \mathfrak{L}_{\maxi}$ with $\codim L_2 = 2$ and $\codim L_4 = 4$. We denote the corresponding minimal subgroups by $C$ and $P$ in accordance with Proposition \ref{prp:characterization_minimal_subgroup_Top}. The claim follows from the subsequent lemma.

\begin{lem}\label{lem:4_and_2}
The subgroups $C$ and $P$ of $G$ act in orthogonal spaces. That is, we have $L_2^{\bot} \bot L_4^{\bot}$.
\end{lem}
\begin{proof} For every nontrivial $r \in C$ the intersection $L_4^{\bot} \cap r L_4^{\bot}$ is nontrivial because of $\mathrm{dim}(L_4^{\bot})=4$ and $\mathrm{codim}(\mathrm{Fix}(r))=2$. Hence $L_4^{\bot}$ is not orthogonal to $r L_4^{\bot}$. Therefore, by Lemma \ref{lem:poincares_are_orthogonal} we can assume that $C$ leaves $L_4^{\bot}$ and $L_4$ invariant. Uniqueness of isotypical (canonical) components of the action of $C$ \cite[Thm.~2.6.8]{MR0450380} yields orthogonal decompositions
\[
		  L_4^{\bot}= \left(L_2^{\bot} \cap L_4^{\bot}\right) \oplus \left(L_2 \cap L_4^{\bot}\right) \text{ and } L_4= \left(L_2^{\bot} \cap L_4\right) \oplus \left(L_2 \cap L_4\right).
\]
Therefore we have an orthogonal decomposition
\[
L_2^{\bot}= \left(L_2^{\bot} \cap L_4^{\bot}\right) \oplus\left( L_2^{\bot} \cap L_4\right).
\]
The subspace $L_2^{\bot} \cap L_4$ is nontrivial, because otherwise we would have $L_4\subseteq L_2$ in contradiction to the maximality of $L_4$.

In order to show that $L_2^{\bot} \bot L_4^{\bot}$ it remains to exclude the case that $L_2^{\bot}\cap L_4^{\bot}$ is one-dimensional. In this case $C$ has order $2$ since $C$ acts nontrivially and freely on the unit sphere in $L_2^{\bot}\cap L_4^{\bot}$. More precisely, the only rotation $r\in C$ is a product of a reflection $s_1$ that inverts $L_2^{\bot}\cap L_4^{\bot}$ and fixes $L_4$ pointwise and another reflection $s_2$ within $L_2^{\bot}\cap L_4$ that fixes $L_4^{\bot}$ pointwise. Since the action of $P$ on $L_4^{\bot}$ is irreducible, there exists some $g \in P$ with $gs_1 g^{-1} \neq s_1$. For, otherwise $\mathrm{Fix}(s_1)^{\bot}$ would be a proper $P$-invariant subspace of $L_4^{\bot}$. The transformations $g$ and $s_2$ commute since they act in orthogonal spaces.
Since also $s_1$ and $s_2$ commute by the same reason, the transformation $r':=r g r g^{-1}=s_1s_2 g s_1s_2 g^{-1}=s_1 g s_1 g^{-1}$ is a rotation contained in $G$ with $L_4\subset\mathrm{Fix}(s_1)\cap\mathrm{Fix}(g)\subset\mathrm{Fix}(r')$. This implies $r'\in I(L_4)=P$. However, a Poincar\'e group does not contain rotations since it acts freely on the unit sphere in $L_4^{\bot}$. Hence, we have $L_2^{\bot} \cap L_4^{\bot}=\{0\}$ by contradiction. The orthogonal decomposition above implies that $L_2^{\bot} \bot L_4^{\bot}$ and so the lemma is proven.
\end{proof}
Let us summarize in a proposition what has been shown so far.
\begin{prp} Let $G<\Or_n$ be a nontrivial finite subgroup. Assume that $\R^n/G$ is a homology manifold. Then $G_{\mini}$ is a nontrivial normal subgroup of $G$ and $G_{\mini}$ splits as a product
\begin{eqnarray*}
		G_{\mini} = G_{\rota} \times P_1 \times \cdots \times P_k
\end{eqnarray*}
of a rotation group $G_{\rota}<\SOr(V_{\rota})$ and Poincar\'e groups $P_i<\SOr(V_{i})$, $i=1,\ldots,k$, that act in pairwise orthogonal spaces $V_{\rota}\oplus V_{1} \oplus \cdots \oplus V_{k}=\R^n$.
\label{prp:minsub}
\end{prp}

\subsection{Free actions on homology spheres} \label{sub:B_and_C_man}
We still assume that $G < \Or_n$ is a finite subgroup for which $\R^n/G$ is a homology manifold. The aim of this section is to show that $G_{\mini}$ actually coincides with $G$ and so to deduce the only if directions of Theorem A and Corollary B in the case of manifolds without boundary.

\begin{lem}Let $\Gamma \curvearrowright X$ be a group action and let $N \triangleleft \Gamma$ be a normal subgroup. If $\Gamma_x = N_x$ for all $x\in X$, then the action
\[
		\Gamma/N \curvearrowright X/N
\]
is free.
\label{lem:free-action}
\end{lem}
\begin{proof} Assume there are $g \in \Gamma$, $h \in N$ and $x \in X$ such that $g x = h x$. Then $h^{-1}g \in \Gamma_x = N_x \subset N$ and so $g \in N$.
\end{proof}

Now we can complete our proof.

\begin{lem} \label{lem:gamma_gamma_mini}Let $G<\Or_n$ be a finite subgroup such that $\R^n/G$ is a homology manifold. Then we have $G = G_{\mini}$.
\end{lem}
\begin{proof}
The proof is by induction on $n$. By Lemma \ref{lem:min_implies_hom_sphere} we have $G<\SOr_n$. For $n=2$ and $n=3$ all finite subgroups of $\SOr_n$ are rotation groups. In this case the claim follows from Proposition \ref{prp:minsub_generators}. So let $n\geq 4$ be fixed and assume that the claim holds in all dimensions lower than $n$. We are going to show that it also holds in dimension $n$.

Let $x \in S^{n-1}$. Since $\R^n/G$ is a homology manifold, so is $(\R x)^{\bot} / G_x=T_x S^{n-1}/G_x$ as discussed in the proof of Lemma \ref{lem:min_implies_hom_sphere}. Hence, we have $G_x = (G_x)_{\mini}$ by induction. Together with Lemma \ref{lem:stabilizer_minimizer_inclusion} we obtain
\[
		G_x = (G_x)_{\mini} \subseteq (G_{\mini})_x \subseteq G_x,
\]
and thus $G_x= (G_{\mini})_x$ for all $x \in S^{n-1}$. Therefore, the action $G/G_{\mini} \curvearrowright S^{n-1}/G_{\mini}$ is free by Lemma \ref{lem:free-action}. Moreover, $\R^n/G_{\mini}$ is a homology manifold by Proposition \ref{prp:minsub} and the if direction of Theorem A, which has been proven at the beginning of this section. Therefore, by Proposition \ref{prp:char_sl_homology} we either have $G=G_{\mini}$, or $n=4$ and $G/G_{\mini}$ is a perfect group.

We can assume that we are in the second case. Then, by Proposition \ref{prp:minsub}, $G_{\mini}$ is either a Poincar\'e group $P$ or the rotation group $G_{\rota}$. If $G_{\mini}$ is a Poincar\'e group, then the action of $G$ on $S^3\subset \R^4$ must be free. For, otherwise there would be a minimal subgroup $H<G_{\mini}=P$ with $\dim L(H)<4$ contradicting the fact that the action of $P$ on $S^3\subset \R^4$ is free. Therefore, in this case we have $G_{\mini}=P=I(L(P))=I(\{0\})=G$. On the other hand, if $n=4$, $G_{\mini}=G_{\rota}$ and $G/G_{\mini}$ is nontrivial and perfect, then Lemma \ref{lem:normal_subgroup_generated_by_ps_in_dim4} implies that $G=P$ and $G_{\rota}=\{1\}$. This is a contradiction to $G_{\mini}\neq \{1\}$ by Lemma \ref{lem:min_nonempty}. Consequently, in any case we have $G = G_{\mini}$ and so the claim follows by induction.
\end{proof}

\begin{proof}[Proof of Theorem A for manifolds without boundary and of Corollary B] At the beginning of this section we have proven the if directions of Theorem A and Corollary B. Recall that a topological manifold is also a homology manifold. Hence, by Lemma \ref{lem:characterization_of_homology_manifolds}, $(ii)$, Lemma \ref{lem:gamma_gamma_mini} and Proposition \ref{prp:minsub} it remains to prove that the additional condition on $k$ depending on $n$ is necessary. This amounts to noting that neither the quotient of $\R^4$ by a Poinca\'e group $P$, nor the suspension $\Sigma (S^3/P) \cong S^0 * S^3/P$ is a topological manifold. This completes the proof of Corollary B and Theorem A in the case of manifolds without boundary. 
\end{proof}

\subsection{The case of manifolds with boundary} \label{sub:boundary_case}

In this section we complete the proof of Theorem A in the general case. First, we prove some preparatory statements.

\begin{lem}\label{lem:hom_man_reflection}
Let $G<\Or_n$ be a finite subgroup. Suppose that $\R^n/G$ is a homology manifold with boundary. Then the boundary of $\R^n/G$ is nonempty if and only if $G$ contains a reflection.
\end{lem}
\begin{proof} By Lemma \ref{lem:ac_hom_quo} we have $G<\SOr_n$ if the boundary of $\R^n/G$ is empty. In particular, the group $G$ does not contain a reflection in this case.

We prove the converse implication by induction on $n$. For $n=1$ the claim is clear. Assume it holds for some fixed $n\geq1$. Let $G<\Or_{n+1}$ be a finite subgroup for which $\R^{n+1}/G$ is a homology manifold with nonempty boundary. Then $S^n/G$ is also a homology manifold with nonempty boundary by Lemma \ref{lem:characterization_of_homology_manifolds}. Let $x \in S^n$ be a point whose coset $\overline{x}$ lies in the boundary of $S^n/G$. Since a neighborhood of $\overline{x}$ is homeomorphic to $T_x S^n /G_x$ the group $G_x\subset G$ contains a reflection by the induction assumption. This completes the proof by induction.
\end{proof}

\begin{cor}\label{cor:boundary_of_hom_man}
Let $G<\Or_n$ be a finite subgroup. Suppose that $\R^n/G$ is a homology manifold with boundary. Then a point $x \in \R^n/G$ belongs to the boundary of $\R^n/G$ if and only if the local group $G_x$ contains a reflection. In other words, by Lemma \ref{lem:charact_codim_1_stratum} the boundary of $\R^n/G$ as a homology manifold coincides with the closure of the codimension $1$ stratum of $\R^n/G$ as an orbifold.
\end{cor}
\begin{proof} A neighborhood of $x$ in $\R^n/G$ is homeomorphic to $\R^n/G_x$. Since the boundary of $\R^n/G_x$ is nonempty if and only of the coset of the origin belongs to the boundary, the claim follows from Lemma \ref{lem:hom_man_reflection}. 
\end{proof}

\begin{lem}\label{lem:double_of_hom_with_boundary}
Let $G<\Or_n$ be a finite subgroup and let $\Gp=G\cap \SOr_n$ be the orientation preserving subgroup of $G$. Suppose that $\R^n/G$ is a homology manifold with nonempty boundary. Then $\R^n/\Gp$ is equivariantly isometric to $2 (\R^n/G)$. In particular, $\R^n/\Gp$ is a homology manifold.
\end{lem}
\begin{proof} By Corollary \ref{cor:boundary_of_hom_man} and Proposition \ref{prp:orbifold_double} the topological double $2(\R^n/G)$ of $\R^n/G$ along its boundary as a homology manifold is homeomorphic to the double of $\R^n/G$ as an orbifold. As mentioned in Section \ref{sub:homology_manifold} the double $2(\R^n/G)$ is a homology manifold \cite[Ch.~5]{MR1402473}. Moreover, the natural projection $2(\R^n/G) \To \R^n/G$ is a covering of Riemannian orbifolds (see Proposition \ref{prp:orbifold_double}). Therefore, by the covering space theory of (Riemannian) orbifolds there exists an index $2$ subgroup $\tilde{G}$ of $G$ for which the identity map of $\R^n/G$ lifts to an isometry from $2(\R^n/G)$ to $\R^n/\tilde{G}$ (see e.g. \cite{Lange2}). Since the acting group is of order two, this isometry must be equivariant. Moreover, since $\R^n/\tilde{G}\cong 2(\R^n/G)$ is a homology manifold, we have $\tilde{G}<\SOr_n$ by Lemma \ref{lem:ac_hom_quo}.  This implies $\tilde{G}\subseteq \Gp$ and thus $\tilde{G}= \Gp$ as both $\tilde{G}$ and $\Gp$ are index $2$ subgroups of $G$. This completes the proof of the lemma.
\end{proof}

Now we can complete the proof of Theorem A. 

\begin{proof}[Proof of Theorem A] It remains to prove the only if direction of Theorem A for manifolds with nonempty boundary. Suppose that $G<\Or_n$ is a finite subgroup for which $\R^n/G$ is a topological-, and hence a homology manifold with nonempty boundary. Then $G$ contains a reflection $s$ by Lemma \ref{lem:hom_man_reflection}. The group $G$ is generated by $s$ and by its orientation preserving subgroup $\Gp$. By Lemma \ref{lem:double_of_hom_with_boundary} the quotient $\R^n/\Gp$ is a homology manifold. In particular, $\Gp$ is a product of a rotation group and a certain number of Poincar\'e groups that act in orthogonal subspaces by the manifold version of Theorem A. Since $s$ normalizes the group $\Gp$ and has a codimension $1$ fixed-point subspace, it can only act in one of these orthogonal subspaces defined by $\Gp$. We have to show that it can only act in the subspace corresponding to the rotation group. Suppose it acts in a four-dimensional subspace defined by a Poincar\'e group. We are going to obtain a contradiction by showing that $\R^n/G$ is not a homology manifold with boundary in this case. By Lemma \ref{lem:characterization_of_homology_manifolds} we can assume that $n=4$, that $\Gp=P$, a Poincar\'e group, and that $G=\left\langle P,s \right\rangle$. Note that the normalizer of $P$ in $\Or_4$ indeed contains reflections.

In this case the coset of $s$ in $G/P$ acts as an orientation reversing isometry on $S^3/P$. Hence, its fixed-point subspace is a disjoint union of isolated points and embedded surfaces. By Corollary \ref{cor:boundary_of_hom_man} the cosets of the isolated fixed points cannot lie in the boundary of $S^3/G$. Moreover, by Lemma \ref{lem:ac_hom_quo} all local groups of $S^3/G$ as an orbifold corresponding to points in the interior of $S^3/G$ as a homology manifold preserve the orientation. Therefore the action of the involution $s$ on $G/P$ cannot have isolated fixed points at all. The remaining fixed embedded surfaces project onto the boundary of $S^3/G$ by Lemma \ref{cor:boundary_of_hom_man}. Now Lemma \ref{lem:characterization_of_homology_manifolds}, $(iv)$, shows that only a single embedded sphere $S^2$ can occur as fixed point set. Therefore, by Lemma \ref{lem:double_of_hom_with_boundary} we can write $S^3/P$ as a double $2_{S^2} (S^3/G)$. The theorem of Seifert and van Kampen on fundamental groups \cite[Thm.~1.20]{MR1867354} implies that $\pi_1(S^3/P)\cong P$ is a free product of isomorphic groups. Such a free product is either trivial or has infinite order. This gives the desired contradiction. 

It remains to show that the additional condition on $k$ depending on $n$ is satisfied. More precisely, we have to show that, for a Poincar\'e group $P<\SOr_4$, the quotient $\R_{\geq0} \times \R^4/P$ is not a topological manifold with boundary. To this end we observe that the boundary of $\R_{\geq 0} \times \R^4/P$ as a homology manifold is $\{0\} \times \R^4/P$, which is not a topological manifold. Therefore also $\R_{\geq 0} \times \R^4/P$ cannot be a topological manifold. Hence, in the case that $n\leq 5$ and that $G$ contains a reflection we must have $k=0$ as claimed. This completes the proof of Theorem A.
\end{proof}

\section{When is a Riemannian orbifold a Lipschitz manifold?} \label{sec:Lip_sec}

We say that a metric space $X$ is a \emph{Lipschitz $n$-manifold with boundary} if each point has a neighborhood that is bi-Lipschitz homeomorphic to some open set in the half-space $\R_{\geq 0}\times \R^{n-1}$. The \emph{boundary} of $X$ is defined as in the case of a topological manifold. By Lemma \ref{lem:lipschitz_orbifold} the question of when a Riemannian orbifold is a Lipschitz manifold with boundary amounts to proving Corollary C, which we will do in the present section. To this end we first recall the concept of a polyhedral metric.

Let $K$ be a locally finite simplicial complex. A compatible choice of flat metrics on the simplices of $K$ induces a so-called \emph{polyhedral metric} on $K$ as the unique maximal metric on $K$ that is majorized by the metrics $d_{\Delta}$ on each simplex $\Delta$ of $K$ \cite[Def.~3.2.4]{MR1835418}. With respect to this metric $K$ is a length space \cite[3.1.24]{MR1835418}. Equivalently, this metric can be described as a \emph{gluing metric} where the distance between two points $x,y \in K$ is defined as the infimum of $	\sum_{i=0}^{k} d_{\Delta_i}(x_i,y_i)$ over all sequences $x_i,y_i$, $i=0,\ldots,k$ with $x=x_0$, $y=y_k$, $x_i,y_i \in \Delta_i$, $\Delta_i$ a simplex of $K$, and $x_i=y_{i+1}$, $i=0,\ldots,k-1$ \cite[3.1.12, 3.1.27]{MR1835418}.

The following lemma shows that the question if the resulting metric space is a Lipschitz manifold does not depend on the specific choice of the polyhedral metric.

\begin{lem}\label{lem:poly_metric_Lip_equiv}
Let $d_0$ and $d_1$ be polyhedral metrics on a locally finite simplicial complex $K$. Then the identity map between $(K,d_0)$ and $(K,d_1)$ is locally bi-Lipschitz. If $K$ is a finite complex, then it is (globally) bi-Lipschitz.
\end{lem}
\begin{proof} The claim follows easily from the observation that for each $\Delta \in K$ there exists some $c>0$ for which the identity map $(\Delta,d_{\Delta,i}) \To (\Delta,d_{\Delta,i})$ is $c$-bi-Lipschitz.
\end{proof}

Given the correctness of the 3-dimensional Poincar\'e conjecture \cite{Per02,Per03a,Per03b}, Siebenmann and Sullivan established the following necessary and sufficient condition for a simplicial complex to be a Lipschitz manifold \cite{MR537747} (cf. \cite[Remark~$(i)$, p.~507]{MR537747}). 
\begin{thm}[Siebenmann-Sullivan] \label{thm:characterization_lip_manifolds}
Let $K$ be a locally finite simplicial complex with some polyhedral metric. Then $K$ is a Lipschitz $n$-manifold, if and only if the link of every vertex of $K$ with some polyhedral metric is in turn a Lipschitz $(n-1)$-manifold and has the homotopy type of an $(n-1)$-sphere.
\end{thm}

In order to apply this result we would like to relate the quotient metric on $\R^n/G$, for a finite subgroup $G < \Or_n$, to a certain polyhedral metric. With respect to the quotient metric $d_q$ on $\R^n/G$ the projection from $\R^n$ to $\R^n/G$ is $1$-Lipschitz. We choose an admissible triangulation $K$ of $\R^n$ for the action of $G$ as discussed in Section \ref{sub:ad_triangulations}. Let $L$ be the link of the origin in $K$ so that $L/G$ is the link of the coset of the origin in $K/G$, cf. Section \ref{sub:ad_triangulations}. Since the action of $G$ on $\R^n$ commutes with multiplication by real numbers, the quotient $\R^n/G$ inherits a cone structure from $\R^n$. For a simplex $\Delta \in L/G$ and a preimage $\hat{\Delta} \in L$ of $\Delta$ we denote the flat metric on $C \Delta$ inherited from $C \hat {\Delta} \subset \R^n$  by $d_{\Delta}$. The metrics $d_{\Delta}$ on the cones $C \Delta$, $\Delta \in L/G$, induce a ``polyhedral'' metric $d_p$ on $\R^n/G$ in the same way as above (we may also assume that the cones $C \Delta$ are subcomplexes of $K/G$ and omit the quotation marks). We work with the cones $C \hat {\Delta}$, $\Delta \in L/G$, instead of the simplices of $K/G$ in order to obtain a \emph{global} Lipschitz homeomorphism in the statement of Corollary C.

\begin{lem}\label{lem:quotient_poly_metric}
The quotient metric $d_q$ and the polyhedral metric $d_p$ on $\R^n/G$ coincide. 
\end{lem}
\begin{proof} Since the projection map $\R^n \To (\R^n/G,d_q)$ is $1$-Lipschitz, on each cone $C\Delta$, $\Delta \in L/G$, the quotient metric $d_q$ is majorized by the metric $d_{\Delta}$ on $C\Delta$ inherited from $\R^n$. By maximality of $d_p$ we thus have $d_q \leq d_p$.

Let $x,y \in \R^n/G$ be two points. The distance $d_q(x,y)$ between the respective orbits in $\R^n$ is realized by a straight line $\gamma:[0,L]\To \R^n$ between representatives of these orbits. Since $\R^n \To (\R^n/G,d_q)$ is $1$-Lipschitz, $\gamma$ projects to a length minimizing path $\overline{\gamma}$ between $x$ and $y$. 
There exists a finite subdivision $0=t_0<t_1<\ldots <t_{k-1}<t_k=L$ such that each restriction $\gamma_{|[t_i,t_{i+1}]}$, $i=0,\ldots,k-1$, is contained in a cone $C \Delta$ of a simplex $\Delta \in L$. Hence, also each restriction $\overline{\gamma}_{|[t_i,t_{i+1}]}$, $i=0,\ldots,k-1$, is contained in a cone $C \Delta_i$, $\Delta_i \in L/G$. It follows that 
\[
			d_p(x,y) \leq\sum_{i=0}^{k-1} d_{\Delta_i}(\overline{\gamma}(t_i),\overline{\gamma}(t_{i+1})))=L(\gamma)=d_q(x,y)
\]
by the characterization of $d_p$ as a gluing metric and hence $d_q=d_p$ as claimed. \end{proof}

In order to treat the case of Lipschitz manifolds with boundary we need one more lemma.

\begin{lem}\label{lem:double_is_Lipschitz} Let $(X,d)$ be a metric space that is a Lipschitz manifold with boundary. Then its metric double $(2X,d_2)$ of $X$ along its boundary (cf. Section \ref{sec:Riemannian_orbifolds}) is a Lipschitz manifold with empty boundary.
\end{lem}
\begin{proof} By passing to the induced length metric (see \cite[2.2.3]{MR1835418}) we can assume that $(X,d)$ is a length space. We leave out the details since we only apply the lemma to a length space anyway (see Section \ref{sec:if_cor_c}). Let $x_0 \in \partial X$ and let $\varphi_0 : \R^{n-1} \times \R_{\geq 0} \To X$ be a locally bi-Lipschitz chart with $x_0 = \varphi(0)$. There exists a unique extension $\varphi : \R^{n} \To 2X$ of $\varphi_0$ that is equivariant with respect to the reflections of $\R^{n}$ at $\R^{n-1} \times \{0\}$ and of $2X$ at $\partial X$. The two copies of $X$ are isometrically embedded in $2X$. Moreover, for some $R>0$ the distance between two points $x,y \in B_{R/8}(x_0)$ that lie in different halfspaces of $2X$ satisfies
\begin{equation}\label{eqn:infinum_rel}
		d_2(x,y)= \inf_{z\in B_R(x_0)\cap \partial X} (d(x,z)+d(z,y)).
\end{equation} Now standard arguments involving the triangle inequality show that $\varphi$ is bi-Lipschitz in a neighborhood of $0$ (cf. older arXiv version of this paper for more details).
\end{proof}

\subsection{Proof of Corollary C} \label{sec:if_cor_c} Now we can prove Corollary C. We begin with the if direction.

\begin{proof}[Proof of the if direction of Corollary C] Let $G<\Or_n$ be a rotation group. By \cite{Lange1} (cf. Theorem \ref{thm:theorem_Lange}) the quotient space $\R^n/G$ is piecewise linear homeomorphic to $\R^n$. As explained in Section \ref{sub:ad_triangulations}, this means that there exists some admissible triangulation $K/G$ of $\R^n/G$ that is simplicially isomorphic to a piecewise linear triangulation $K'$ of $\R^n$. We can assume that the coset $\overline{0}$ of the origin in $\R^n/G$ is identified with the origin in $\R^n$. Then the simplicial isomorphism between $K/G$ and $K'$ restricts to a simplicial isomorphism $\phi$ between the links $L:=\mathrm{link}_{K/G}(\overline{0})$ and $L':=\mathrm{link}_{K'}(0)$. Since the complexes $L$ and $L'$ are finite, it follows (similarly as in the proof of Lemma \ref{lem:poly_metric_Lip_equiv}) that there exists some $c>0$ such that the linear extension $\phi: \R^n/G=C(L)\To \R^n=C(L')$ is $c$-bi-Lipschitz with respect to the polyhedral metrics induced by the flat cones $C\Delta$ and $C\Delta'$, $\Delta \in L$, $\Delta' \in L'$. Since the quotient metric and the polyhedral metric on $\R^n/G$ coincide by Lemma \ref{lem:quotient_poly_metric}, the claim follows in this case.
%
%

If $G<\Or_n$ is a reflection-rotation group that contains a reflection, then the quotient space $\R^n/G$ is piecewise linear homeomorphic to $\R_{\geq 0} \times \R^{n-1}$ by \cite{Lange1} (cf. Theorem \ref{thm:theorem_Lange}). In this case it follows in the same way as above that the quotient space $\R^n/G$ is bi-Lipschitz homeomorphic to $\R_{\geq 0} \times \R^{n-1}$.
\end{proof}

\begin{proof}[Proof of the only if direction of Corollary C] Let us first prove the statement for manifolds without boundary. In this case a proof can be given in the same spirit as in the piecewise-linear category that works by induction on $n$, cf. \cite[Sect.~4]{Lange1}. For $n\leq 2$ the claim is trivially true. Assume it holds for some $n\geq 2$ and let $G<\Or_{n+1}$ be a finite subgroup for which $\R^{n+1}/G$ is a Lipschitz manifold. By Section \ref{sub:ad_triangulations} we can choose a triangulation $K$ of $\R^{n+1}$ that is admissible for the action of $G$. By Theorem \ref{thm:characterization_lip_manifolds} the link $L/G=\mathrm{link}_{K/G}(\overline{0})$, $L=\mathrm{link}_K (0)$, of the image $\overline{0}$ of $0\in\R^n$ in $K/G$ is a Lipschitz manifold with respect to its induced length metric. Let $v$ be a vertex of $\mathrm{link}_K (0)$ and let $\overline{v}$ be its equivalence class in $L/G$. Radial projection induces a bi-Lipschitz homeomorphism between a neighborhood of $\overline{v}$ in $L/G$ and a neighborhood of the coset of the origin in $T_v S^n /G_v$. In particular, for any vertex $v$ of $\mathrm{link}_K( 0)$ the isotropy group $G_v$ is a rotation group by our induction assumption.

The group $G_{\mathrm{iso}}$ generated by all such isotropy groups is normal in $G$ since $G$ permutes the vertices of $\mathrm{link}_K (0)$. Moreover, $G_{\mathrm{iso}}$ is a rotation group by construction. We claim that $G_{\mathrm{iso}}$ contains every isotropy group $G_x$ for $x\in \mathrm{link}_K (0)$. Indeed, if $g\in G$ fixes a point $x\in \mathrm{link}_K (0)$, then by regularity of $\mathrm{link}_K (0)$ (see Section \ref{sub:ad_triangulations}), it fixes the smallest dimensional simplex in $\mathrm{link}_K (0)$ that contains $x$ pointwise. In particular, $g$ fixes a vertex of this simplex and is thus contained in $G_{\mathrm{iso}}$. Therefore, for every $x\neq 0$, we have 
\[
	G_x \subseteq G_x \cap G_{\mathrm{iso}}=(G_{\mathrm{iso}})_x \subseteq G_x
\]
and so $(G_{\mathrm{iso}})_x=G_x$. This means that the action of $G/G_{\mathrm{iso}}$ on $S^n/G_{\mathrm{iso}} \cong(\mathrm{link}_K( 0))/G_{\mathrm{iso}}$ is free by Lemma \ref{lem:free-action}. Because of $n\geq 2$, the quotient space $(S^n/G_{\mathrm{iso}})/(G_{\mathrm{iso}}/G)\cong S^n/G \cong (\mathrm{link}_K (0))/G\cong \mathrm{link}_{K/G} (\overline{0}) $ is simply connected by Theorem \ref{thm:characterization_lip_manifolds} and our assumption that $\R^{n+1}/G$ is a Lipschitz manifold. It follows that $G=G_{\mathrm{iso}}$. In particular, $G$ is a rotation group and the claim follows by induction.

Now suppose that $\R^n/G$  is a Lipschitz manifold with non-empty boundary for some finite subgroup $G<\Or_n$. In this case it follows as in the proof of Theorem A at the end of Section \ref{sec:if_dir}, by an application of Lemma \ref{lem:double_is_Lipschitz}, that $G$ contains a reflection and that its orientation-preserving subgroup $\Gp$ is a rotation group. This shows that $G$ is a reflection-rotation group that contains a reflection and thus completes the proof of Corollary C.
\end{proof}

\end{document}